\documentclass[article]{imsart}
\usepackage{amsthm, amssymb, amsmath}
\usepackage{bbm}
\usepackage[toc,page]{appendix}
\usepackage{mathtools, thmtools}
\usepackage{euscript}
\usepackage{color}
\usepackage{tikz}
\usetikzlibrary{arrows}
\usepackage[utf8]{inputenc}
\usepackage[shortlabels]{enumitem}
\usepackage{url}
\usepackage{etex}
\usepackage{alltt}
\usepackage{memhfixc}
\usepackage{bbold}%
\usepackage{cases}
\usepackage{enumitem}
\usepackage[noadjust]{cite}
\let\citep=\cite
\definecolor{labelkey}{rgb}{0.0, 0.8, 0.3}
\usepackage[top=2.5cm,bottom=2.5cm,left=2.5cm,right=2.5cm,includeheadfoot,asymmetric]{geometry}
\usepackage{algorithm}
\usepackage{xfrac}
\usepackage{wrapfig}

\numberwithin{equation}{section}
\usepackage{hyperref}
\hypersetup{
  colorlinks = true,
  urlcolor = blue, 
  linkcolor = blue,
  citecolor = blue}

\usepackage{amsthm}
\usepackage{cleveref}
\crefname{equation}{}{}
\Crefname{equation}{}{}
\usepackage{xcolor,colortbl}
\usepackage{comment}

\declaretheorem[name=Lemma]{lemma}
\declaretheorem[name=Proposition]{proposition}

\declaretheorem[name=Remark, style=remark]{remark}
\declaretheorem[name=Theorem]{theorem}

\newtheorem{assumption}{Assumption}
\newtheorem{corollary}[theorem]{Corollary}

\newcommand\simiid{\stackrel{iid}{\sim}}
\newcommand\TV{\sf{TV}}
\newcommand\one{\mathbbm{1}}
\newcommand\E{\mathbb{E}}
\newcommand\eqdef{\triangleq}

\newcommand{\bb}{\mathbb}

\newcommand\subsetsim{\mathrel{%
  \ooalign{\raise0.2ex\hbox{$\subset$}\cr\hidewidth\raise-0.8ex\hbox{\scalebox{0.9}{$\sim$}}\hidewidth\cr}}}

\renewcommand{\P}{\mathbb{P}}

\newcommand{\w}{\omega}  
\newcommand{\sgn}{\operatorname{sign}}

\renewcommand{\one}{\bbm{1}}
\renewcommand{\sf}{\mathsf}
\renewcommand{\P}{\mathbb P}

\newcommand{\sref}[2]{\hyperref[#2]{#1 \ref*{#2}}}

\renewcommand{\cal}{\mathcal}

\renewcommand{\Im}{\mathfrak{I}}

\newcommand{\mc}{\mathcal}

\newcommand{\D}{\mathrm{d}}
\renewcommand{\d}{\mathrm{d}}
\newcommand{\C}{\mathbb{C}}

\newcommand{\N}{\mathbb{N}}
\newcommand{\Q}{\mathbb{Q}}
\newcommand{\Z}{\mathbb{Z}}

\DeclareMathOperator\vol{vol}
\DeclareMathOperator*{\argmin}{arg\,min}

\DeclareMathOperator{\sign}{sign}

\DeclareMathOperator\supp{supp}
\renewcommand{\d}{\mathrm{d}}

\newcommand\bbm[1]{\mathbbm{#1}}

\newcommand\mbb[1]{\mathbb{#1}}

\newcommand{\R}{\bb{R}}

\renewcommand{\E}{\bb E}

\newcommand{\eps}{\varepsilon}

\renewcommand{\O}{\mathcal{O}}

\renewcommand{\sf}{\mathsf}
\renewcommand{\P}{\mathbb P}
\newcommand{\EE}{\mathbb E}

\newcommand{\mmd}{\operatorname{MMD}}
\renewcommand{\cal}{\mathcal}

\usepackage{xcolor}

\def\braket#1{\langle#1\rangle}

\renewcommand{\abstractname}{Abstract}
\begin{document}

\begin{frontmatter}

	\title{Density estimation using\\the perceptron}
	\runtitle{Density estimation using the perceptron}
	
\author{Patrik Róbert Gerber \hfill prgerber@mit.edu \\ 
	Tianze Jiang \hfill  tjiang@mit.edu \\ 
        Yury Polyanskiy \hfill yp@mit.edu \\
        Rui Sun \hfill  eruisun@mit.edu \\}

	
\runauthor{Gerber, Jiang, Polyanskiy and Sun}

\begin{abstract}
\small{We propose a new density estimation algorithm. Given
$n$ i.i.d. observations from a distribution belonging to a class
of densities on $\mathbb{R}^d$, our estimator outputs any density in the class whose ``perceptron
discrepancy'' with the empirical distribution is at most $O(\sqrt{d/n})$.
The perceptron discrepancy is defined as the largest
difference in mass two distribution place on any halfspace. It is shown that
this estimator achieves the expected total variation distance to the truth that is almost
minimax optimal over the class of densities with bounded Sobolev norm and Gaussian
mixtures. This suggests that the regularity of the prior distribution could be an
explanation for the efficiency of the ubiquitous step in machine learning that replaces optimization over large function spaces with simpler parametric
classes (such as discriminators of GANs).
We also show that replacing the perceptron discrepancy with
the generalized energy distance of \cite{szekely2013energy} further improves
total variation loss. The generalized energy distance between empirical
distributions is easily computable and
differentiable, which makes it especially useful for fitting generative models.
To the best of our knowledge, it is the first ``simple'' distance with such
properties that yields minimax optimal statistical guarantees. 

In addition, we shed light on the ubiquitous method of representing discrete data in domain $[k]$ via embedding vectors on a unit ball in $\mathbb{R}^d$. We show that taking $d \asymp \log(k)$ allows one to use simple linear probing to evaluate and estimate total variation distance, as well as recovering minimax optimal sample complexity for the class of discrete distributions on $[k]$.
\footnote{Supported in part by the NSF grant CCF-2131115 and the MIT-IBM Watson AI Lab.}

}
\end{abstract}

\end{frontmatter}

\newpage

\tableofcontents
\setcounter{tocdepth}{2}

\newpage
\section{Introduction}\label{sec:intro}

A standard step in many machine learning algorithms is to replace an (intractable) optimization
over a general function space with an optimization over a large parametric class (most often
neural networks). This is done in supervised learning for fitting classifiers, in variational
inference \citep{blei2017variational,zhang2018advances} for applying ELBO, in variational
autoencoders \citep{kingma2019introduction} for fitting the decoder, in Generative Adversarial
Networks (GANs) \citep{goodfellow2014generative,arjovsky2017wasserstein} for fitting the
discriminator, in  diffusion models \citep{song2020score,chen2022sampling} for fitting the score
function, and many other settings.

To be specific, let us focus on the example of GANs, which brought about the new era of density
estimation in high-dimensional spaces. The problem setting is the following. We are given access
to an i.i.d. data $X_1,\dots,X_n \in \R^d$ sampled from an unknown distribution $\nu$ and a class of
distributions $\cal G$ on $\R^d$ (the class of available ``generators''). The goal of the learner is to find $\argmin_{\nu' \in \cal G}
D(\nu', \nu)$, where $D$ is some dissimilarity measure (``metric'') between probability
distributions. In the case of GANs this measure is the Jensen-Shannon divergence
$\mathrm{JS}(p,q) \triangleq \mathrm{KL}(p\|{1\over 2}p + {1\over2} q) + \mathrm{KL}(q\|{1\over 2}p + {1\over2} q)$ where $KL(p\|q) = \int p(x) \log {p(x)\over q(x)} dx$ is the Kullback-Leibler
divergence.
As any $f$-divergence, JS has a variational form (see \cite[Example 7.5]{yuryyihongbook}):
$ \mathrm{JS}(p,q) = \log 2 + \sup_{h:\R^d \to (0,1)} \EE_p[h] + \EE_q[\log(1-h)]\,.$
With this idea in mind, we can now restate the objective of minimizing $\mathrm{JS}(\nu',\nu)$ as a game between
a ``generator'' $\nu'$ and a ``discriminator'' $h$, i.e., the GAN's estimator is
\begin{equation}\label{eq:gan_1}
	\tilde \nu \in \argmin_{\nu'} \sup_{h:\R^d \to (0,1)} {1\over n} \sum_{i=1}^n h(X_i) + \EE_{\nu'}[\log(1-h)]\,,
\end{equation}
where we also replaced the expectation over (the unknown) $\nu$ with its empirical version 
$\nu_n \eqdef
\frac1n\sum_{i=1}^n\delta_{X_i}$. 
Subsequently, the idea was extended to other types of metrics, notably the Wasserstein-GAN
\citep{arjovsky2017wasserstein}, which defines
\begin{equation}\label{eqn:intro adversarial}
    \tilde\nu \in \argmin_{\nu' \in \cal G} \sup_{f \in \cal D} \left|\E_{Y\sim\nu'} f(Y) - \frac1n\sum_{i=1}^n f(X_i)\right|, 
\end{equation}
where the set of discriminators $\cal D$ is a class of Lipschitz functions (corresponding to
the variational characterization of the Wasserstein-1 distance).

The final step to turn~\eqref{eq:gan_1} or~\eqref{eqn:intro adversarial} into an algorithm is to
relax the domain of the inner maximization
(``discriminator'') to a parametric class of neural network discriminators $\cal D$. 
Note that replacing $\sup_{h:\R^d\to (0,1)}$ with $\sup_{h\in \cal D}$ effectively changes the
objective from minimizing the JS divergence to minimizing a ``neural-JS'', similar to how MINE~\citep{belghazi2018mutual}
replaces the true mutual
information with a ``neural'' one. This weakening is quite worrisome for a statistician. While the JS divergence is a strong statistical
distance, as it bounds total variation from above and from
below~\cite[Eq. (7.39)]{yuryyihongbook}, the ``neural-JS'' is unlikely to possess any such properties.

How does one justify this restriction to a simpler class $\cal D$? A practitioner would say that while
taking $\max_{h\in \cal D}$ restricts the power of the discriminator, the design of $\cal D$ is
fine-tuned to picking up those features of the distributions that are relevant to the human
eye.\footnote{Implying in other words, that whether or not total variation $\TV(\tilde \nu, \nu)$ is high is
irrelevant as long as the generated images look ``good enough'' to humans.} A theoretician,
instead, would appeal to universal approximation results about neural networks to claim that
restriction to $\cal D$ is almost lossless.

The purpose of this paper is to suggest, and prove, a third explanation: the answer is in
the \emph{regularity} of $\nu$ itself. Indeed, we show that the restriction of
discriminators to a very small class $\cal D$ in~\eqref{eq:gan_1} results in almost no loss of minimax statistical guarantees, even if $\cal D$ is far from being a universal approximator. That is, the minimizing distribution $\tilde \nu$ selected with respect to a weak form of the distance enjoys almost minimax optimal guarantees with respect to the strong total variation distance, provided
that the true distribution $\nu$ is regular enough.  Phrased yet another way,  even though the ``neural''
distance is very coarse and imprecise, and hence the minimizer selected with respect to it might be expected to only fool very naive discriminators, in reality it turns out to fool any arbitrarily complex, but bounded discriminator.

Let us proceed to a more formal statement of our results. One may consult Section~\ref{sec:notation} for
notation. We primarily focus on two
classes of distributions on $\R^d$: first, $\cal P_S(\beta,d, C)$ denotes the set of distributions
supported on the $d$-dimensional unit ball $\bb B(0,1)$ that have a density with finite $L^2$ norm and whose
$(\beta,2)$-Sobolev norm, defined in \eqref{eqn: def sobolev norm}, is bounded by $C$; second,  $\cal P_G(d) = \{\mu * \cal N(0,1) :
\operatorname{supp}(\mu) \subseteq \bb B(0,1)\}$ is the class of Gaussian mixtures
with compactly supported mixing distribution. We remind the reader that the total variation distance has the variational form 
$\TV(p,q) = \sup_{h: \R^d \to [0,1]} \EE_p h - \EE_q h\,.$ Our first result concerns the following class of discriminators:
\begin{equation}\label{eqn: def D_1}
    \cal D_1 = \{x\mapsto\one\{x^\top v \geq b\} : v \in \R^d, b\in\R\}\,,
\end{equation}
the class of affine classifiers, which can be seen as a single layer perceptron with a threshold
non-linearity. 

\begin{theorem}\label{thm:intro}
For any
$\beta > 0$, $d\geq1$ and $C>0$, there exists a finite constant $C_1 = C_1(\beta, d, C)$: 
\begin{equation}\label{eqn:intro smooth}
        \sup_{\nu \in \cal P_S(\beta,d, C)} \E \TV(\tilde\nu,\nu) \leq C_1 n^{-\frac{\beta}{2\beta+d+1}}\,,
    \end{equation}
where the estimator $\tilde\nu$ is defined in \eqref{eqn:intro adversarial} with $\cal D = \cal D_1$
and $\cal G=\cal P_S(\beta,d, C)$.
Similarly, for any $d\geq1$ there exists a finite constant $C_2 = C_2(d)$ so that
    \begin{equation*}
        \sup_{\nu \in \cal P_G(d)} \E \TV(\tilde\nu,\nu) \leq C_2
	\frac{(\log(n))^{\frac{2d+2}{4}}}{\sqrt n}\,,
    \end{equation*}
where the estimator $\tilde\nu$ is defined in \eqref{eqn:intro adversarial} with $\cal D = \cal D_1$
and $\cal G=\cal P_G(d)$.
\end{theorem}

Recall the classical result ~\citep{ibragimov1983estimation} which shows that the minimax optimal estimation rate in TV over the class $\cal P_S(\beta,d, C)$ equals $n^{-\beta/(2\beta+d)}$ up to
constant factors. Thus, the estimator in \eqref{eqn:intro smooth} is \textit{almost} optimal, the only difference being that the dimension $d$ is replaced by $d+1$.  Similarly, for the Gaussian mixtures we reach the parametric rate up to
a polylog factor.\footnote{For estimation of Gaussian mixtures in total variation the precise value of the minimax
optimal polylog factor is at present unknown. However, for the $L_2$ distance the minimax rate is
known, and in the course of our proofs (see \eqref{eqn:E_g vs L2}) we show that our estimator only loses a 
 multiplicative factor of $\log(n)^{1/4}$ in loss compared to the optimal $L_2$-rate $\log(n)^{d/4}/\sqrt n$ derived
in~\citep{kim2022minimax}.}

The proof of \Cref{thm:intro} relies on a comparison inequality between total variation and the ``perceptron discrepancy'', or maximum halfspace distance, which we define as
\begin{equation}\label{eqn: def d_H}
    \overline{d_H}(\mu,\nu) \eqdef \sup_{f\in\cal D_1} \{\E_\mu f - \E_\nu f\}.
\end{equation} 
Note first that $\overline{d_H} \leq \TV$ clearly holds since all functions in the class $\cal
D_1$ are bounded by $1$. For the other direction, by proving a generalization of the Gagliardo-Nirenberg-Sobolev inequality we derive the following comparisons.

\begin{theorem}\label{thm:intro comparison} For any
$\beta > 0$, $d\geq1$, there exists a finite constant $C_1 = C_1(\beta, d)$:
\begin{equation}\label{eq:compare}
    \TV(\mu,\nu)^{\frac{2\beta+d+1}{2\beta}}\leq C_1\cdot (\|\mu\|^2_{\beta, 2}+\|\nu\|^2_{\beta, 2})^{\frac{d+1}{4\beta}}\cdot  \overline{d_H}(\mu,\nu)
\end{equation}
holds for all $\mu,\nu \in \cal P_S(\beta,d, \infty)$. Similarly, for any $d\geq1$ there exists a finite constant $C_2 = C_2(d)$ such that 
\begin{equation*}
    \,\TV(\mu,\nu)  \log\left(3 + \frac{1}{\TV(\mu,\nu)}\right)^{-\frac{d+1}{2}} 
    \leq  C_2 \overline{d_H}(\mu,\nu) 
\end{equation*}        
holds for all $\mu,\nu \in \cal P_G(d)$. 
\end{theorem}

We remark that we also show (in \Cref{thm:tightness}) that the exponent $2\beta + d+1\over
2\beta$ in~\eqref{eq:compare} is tight, i.e. cannot be improved in general.

With \Cref{thm:intro comparison} in hand the proof of \Cref{thm:intro} is \textit{notably} simple. For example, 
let us prove~\eqref{eqn:intro smooth} (for full details, see~Section~\ref{sec:proof of intro}). Recall that $X_i\simiid \nu$, $\nu_n$ is the empirical
distribution and $\tilde \nu = \argmin_{\nu' \in \cal P_S} \overline{d_H}(\nu', \nu_n)$. We then
have from the triangle inequality and minimality of $\tilde \nu$:
$$ \overline{d_H}(\tilde \nu, \nu) \le  \overline{d_H}(\tilde \nu, \nu_n)  + 
\overline{d_H}(\nu_n, \nu) \le 2\overline{d_H}(\nu_n, \nu)\,.$$
Thus, from \Cref{thm:intro comparison} we have
\begin{equation}\label{eq:is_1}
	\TV(\tilde \nu, \nu) \lesssim \overline{d_H}(\nu_n, \nu)^{\frac{2\beta}{2\beta+d+1}}\,.
\end{equation}
Lastly, we recall that $\cal D_1$ is a class with finite VC-dimension and thus from uniform
convergence (Theorem 8.3.23, \cite{vershynin2018high}) we have for some dimension-dependent constant $C(d)$ that
$$ \EE[\overline{d_H}(\nu_n, \nu)] \le {C(d)\over \sqrt{n}}\,.$$
Thus, applying expectation and Jensen's inequality to~\eqref{eq:is_1} we get 
$$ \EE[\TV(\tilde \nu, \nu)]\lesssim   \EE[\overline{d_H}(\nu_n, \nu)^{\frac{2\beta}{2\beta+d+1}}] \lesssim 
 \EE[\overline{d_H}(\nu_n, \nu)]^{\frac{2\beta}{2\beta+d+1}} \lesssim  n^{-\frac{\beta}{2\beta+d+1}}
$$
as claimed.

While we believe that \Cref{thm:intro} provides theoretical proof for the efficacy of simple discriminators, it has several theoretical and practical deficiencies that we need to address.
First, the guarantee in \Cref{thm:intro} for $\cal P_S$ is strictly worse than the minimax optimal rate, which is ${\O}\left(n^{\frac{-\beta}{2 \beta+d}}\right)$ (see e.g. \cite{ibragimov1983estimation}).

Second, from the implementation point of view, computing the distance $\overline{d_H}$ behind \Cref{thm:intro} is
impractical. Indeed, finding the halfspace with maximal separation
between even two empirical measures is a nonconvex, non-differentiable problem and takes super-poly
time in the dimension $d$ assuming $\mathsf{P}\neq\mathsf{NP}$ \citep{guruswami2009hardness}, and $\omega(d^{\omega(\eps^{-1})})$ time for $\eps$-optimal agnostic learning between two densities assuming either \textsf{SIVP} or \textsf{gapSVP} \citep{tiegel23a}. 

Finally, even if we disregard the computational complexity, it is unclear how to minimize $\tilde\nu$ concerning $\arg\min_{\nu'} \overline {d_H}(\nu',
\nu_n)$ for given samples. This concern is alleviated by the fact that any $\tilde \nu$ satisfying
$\overline {d_H}(\tilde \nu, \nu_n) = \O (\sqrt{d/n})$ will work without degrading our performance guarantee, and thus only an approximate
minimizer is needed. 

To address the above concerns, we make two changes to improve $\overline{d_H}$ in the min-distance density estimator: (a)
we replace the perceptron class $\cal D_1$ in \eqref{eqn: def D_1} with a generalized perceptron $\cal D_\gamma$ for $\gamma \in (0,2)\setminus\{1\}$ defined as:
\begin{equation}\label{eqn: def D_gamma}
    \cal D_\gamma = \{x\mapsto |x^\top v- b|^{\gamma - 1\over 2} : v \in \R^d, b\in\R\}\,,
\qquad \gamma \in (0,2)\setminus\{1\}
\end{equation}
(b) we replace the perceptron discrepancy $\overline{d_H}$ (defined
with respect to the ``best'' perceptron) with an ``\emph{average}'' version $d_H$ defined in~\eqref{eq:dh} (see \eqref{eqn:slice equiv form} for the definition with general $\gamma$). Therefore, one does not even need to find an approximately optimal half-space, as random half-spaces provide sufficient discriminatory power. These changes are made precise in \Cref{sec:equiv formulations}.

Our \Cref{cor:d_a density est} and \Cref{prop:optimize_alpha} show that these two changes allow us to achieve a total variation rate of $n^{-\beta/(2\beta+d+\gamma)}$ for the min distance density estimator. The improved rate comes to (within $\operatorname{polylog}(n)$) minimax optimality as $\gamma \to 0$ adaptively with $n$, addressing our first concern.

Somewhat unexpectedly, we discover that the average perceptron
discrepancy $d_H$ exactly equals
Sz\'ekely and Rizzo's energy distance $\cal{E}_1$ (Definition 1, \cite{szekely2013energy}), defined as  
\begin{equation}\label{eqn:gen energy}
    \cal{E}_1^2(\mu,\nu) \triangleq \mathbb{E}\left[ 2\|X-Y\| - \|X-X'\| -
    \|Y-Y'\|\right], \qquad (X,X',Y,Y')\sim \mu^{\otimes 2}\otimes\nu^{\otimes 2}\,,
\end{equation}
where $\| \cdot \|$ is the usual Euclidean norm on $\R^d$. Thus, our 
\Cref{cor:d_a density est} (with $\gamma=1$) shows that minimizing $\min_{\nu'} \cal E_1(\nu',\nu_n)$ gives a density
estimator with rates over $\cal P_S$ and $\cal P_G$ as given in \Cref{thm:intro}. 
Furthermore, for $\gamma >1$ the corresponding average over $\mathcal{D}_\gamma$ results in the distance $\mathcal{E}_\gamma$ known as \textit{generalized energy distance}, defined in the same paper. See \Cref{sec:equiv formulations} for details.

This discovery addresses our second and final concern in the following sense. Treating our distance $\cal E_\gamma(\widetilde\nu^{\operatorname{gen.}}_m,\nu^{\operatorname{target}}_n)$ as a loss function for solving $\widetilde\nu^{\operatorname{gen.}}_m$, the closed-form computation of $\mathcal{E}$ via \eqref{eqn:gen energy} requires only a polynomial
$O(n^2+m^2)$ steps, and is friendly to gradient evaluations. 

Overall, our message from an algorithmic point of view is as follows: assuming one has access to a parametric family of generators sampling from $\nu_\theta$ for parameters $\theta \in \mathbb{R}^p$, and if one can compute $\nabla_\theta$ of the generator forward pass, e.g., via pushforward of a reference distribution under a smooth transport map or neural network-based models (\cite{wang2022minimax, marzouk2023distribution}), then one can fit $\theta$ to the empirical sample $\nu_n$ by running stochastic gradient descent steps:
\begin{itemize}
    \item sample $m$ samples from $\nu_\theta$ and form the empirical distribution $\nu_m^{\prime}$,
    \item compute the loss $\mathcal{E}_\gamma\left(\nu_m^{\prime}, \nu_n\right)$ and backpropagate the gradient with respect to $\theta$,
    \item update $\theta \leftarrow \theta-\eta \nabla_\theta \mathcal{E}_\gamma\left(\nu_m^{\prime}, \nu_n\right)$ for some step size $\eta$.
\end{itemize}
Again, computing $\cal E_\gamma(\nu'_m,\nu_n)$ via \eqref{eqn:gen energy} requires
$O(n^2+m^2)$ steps and is gradient-friendly. 
Note, though, that obtaining and certifying good parameterizations for generators with densities satisfying Sobolev norm constraints is nontrivial and may require careful regularization.
In this work, however, our focus will be on the discriminator part of GANs, as opposed to the generator.

\subsection{Contributions}\label{sec: contributions}
    This work focuses on studying the discriminator classes for adversarially estimating smooth densities, our main contributions are as follows. We show that $\beta$-smooth distributions, Gaussian mixtures and discrete distributions that are far apart in total variation distance must possess a halfspace on which their mass is substantially different (\Cref{thm:intro comparison,prop:d_a comparison,prop:discrete comparison}). 
    
     We apply the separation results to density estimation problems, showing that an ERM density estimator nearly attains the minimax optimal density estimation rate with respect to $\TV$ over the aforementioned distribution classes (\Cref{thm:intro,cor:d_a density est,thm:discrete estimation}), suggesting that halfspaces indicators are (almost) sufficient for discriminator classes in GAN \eqref{eqn:intro adversarial} in order to achieve (close to) minimax optimal density estimation in the considered classes of distributions. 

     In \Cref{sec:equiv formulations} we show that the average halfspace separation distance $d_H$ is equal up to constant to the energy distance $\cal{E}_1$ (\Cref{prop:dh_energy}), which has many equivalent expressions: as a weighted $L^2$-distance between characteristic functions (\Cref{prop:norm equiv form}), as the sliced Cramér-$2$ distance (\Cref{prop:slice equiv form}), as an IPM/MMD/energy distance (\Cref{sec:mmd form}), and as the $L^2$-norm of the Riesz potential (\Cref{prop:riesz}).
     
     We generalize the average halfspace distance $d_H$ to include an exponent $\gamma\in(0,2)$, corresponding to the generalized energy distance $\cal E_\gamma$. Consequently, we discover that if instead of thresholded linear features $\one\{ v^\top x > b\}$ we use the non-linearity $|v^\top x - b|^\gamma$, smooth distributions and Gaussian mixtures can be separated even better (\Cref{prop:d_a comparison}). Combined with the fact that $\cal E_\gamma$, similarly to $d_H$, decays between population and sample measures at the parametric rate (\Cref{lem:d_a concentration}), the ERM for $\cal E_\gamma$ reduces the slack in the density estimation rate, achieving minimax log-optimality (\Cref{prop:optimize_alpha}). This result, combined with its strong approximation properties, supports its use in modern generative models (e.g. \cite{ho2020denoising,goodfellow2014generative,rombach2022high,ramesh2022hierarchical}). 

     Finally, \Cref{prop: tst bad} shows that recent work applying $\overline{d_H}$ for two-sample testing is sub-optimal over the class of smooth distributions in the minimax sense.

\subsection{Related Work}\label{sec:related works}
An important concept related to our work is the Integral Probability Metrics (IPMs) 
\begin{equation}\label{eqn:def_IPM}
    d^{\text{IPM}}_{\cal D}(\P,\Q)\eqdef \sup_{f\in\cal D}|\E_\P f - \E_\Q f|
\end{equation}
where $\mathcal{D}$ is the discriminator class. Examples include $\TV$ when $\mathcal{D}$ is the class of bounded functions and $W_1$ when $\cal D$ is the class of Lipschitz functions. IPM distances represent the max separation between two distributions one gets from an adversary powered with $\mathcal{D}$, and can be viewed as the generator objective in GAN \eqref{eqn:intro adversarial}.

In terms of distance comparison inequalities on smooth density classes similar to our
\Cref{thm:intro comparison}, the closest work we found was \cite{CHAE2020108771}, in which the
authors have shown comparison between $\TV\lesssim W_1^{{\beta}/{(\beta+1)}}$ for $(L_1, \beta)$
Sobolev smooth densities on $\mathbb{R}$. This inequality is also optimal in the exponent.
Subsequently, \cite{chae2024wasserstein} extended comparison inequality to Besov densities on
$\mathbb{R}^d$ and $W_p$ metrics. These results\footnote{We thank anonymous reviewers for pointing
them to us.} serve the same purpose as ours: they show that achieving optimal estimation in a
``weak'' norm (such as $W_1$) implies optimal estimation rate under a ``strong'' norm. Since $W_1$
corresponds to a non-parametric discriminator class of all Lipschitz functions, this result 
does not explain why simple discriminators work so well in practice. 

Another paper related to distance comparison is \cite{bai2018approximability} whose authors study comparison between the Wasserstein distance $W_1$ and the IPM $d_{relu}$ defined by the discriminator class $\cal D=\{x\mapsto \operatorname{Relu}(x^\top v+b):b,\|v\|\leq1\}$. They show \cite[Theorem 3.1]{bai2018approximability} that $\sqrt{\kappa/d} W_1 \lesssim d_{relu} \lesssim W_1$ for Gaussian distributions with mean in the unit ball, where $\kappa$ is an upper bound on their condition numbers and $d$ is the dimension. They obtain results for other distribution classes (Gaussian mixtures, exponential families), but for each of these they use a different class of discriminators that is adapted to the problem. 

In the density estimation literature, an estimator of the form \eqref{eqn:intro adversarial} applied on smooth densities first appeared in the famous work of \cite{yatracos1985rates}. Instead of indicators of halfspaces, they consider the class of discriminators $$\cal Y_\epsilon \eqdef \{\one\{\D\nu_i/\D\nu_j \geq 1\}:1 \leq i,j \leq N(\epsilon, \cal G)\},$$where $\nu_1, \dots, \nu_{N(\epsilon, \cal G)}$ forms a minimal $\epsilon$-$\TV$ covering of the class $\cal G$ and $N(\epsilon, \cal G)$ is the so-called covering number. 
Writing $d_Y(\mu,\mu') = \sup_{f\in\cal Y} (\E_\mu f-\E_{\mu'} f)$, it is not hard to prove that $|\TV - d_Y| = O(\epsilon)$ on $\cal G\times \cal G$ and that $\E d_Y(\nu,\nu_n) \lesssim \sqrt{\log N(\epsilon_n, \cal G)/n}$ by a union bound coupled with a binomial tail inequality. From here $\E \TV(\tilde\nu,\nu) \lesssim \sqrt{\log N(\epsilon, \cal G)/n} + \epsilon$ follows by the triangle inequality (here $\tilde\nu$ is defined as in \eqref{eqn:intro adversarial} with $\cal D = \cal Y$). Note that in contrast to our perceptron discrepancy $\overline{d_H}$, Yatracos' estimator attains the optimal rate on $\cal G = \cal P_S$, corresponding to the choice $\epsilon=\epsilon(n) \asymp n^{-\beta/(2\beta+d)}$. 

In terms of GAN (or more generally any adversarial) density estimation, a series of papers 
(\cite{singh2018nonparametric}, 
\cite{uppal2019nonparametric},
\cite{chen2020distribution},
\cite{belomestny2021rates},
\cite{liang2021well},
\cite{chae2022rates},
\cite{stephanovitch2023wasserstein},
etc)
study the problem of density estimation over smoothness classes with respect to IPM distances.
These works typically choose the discriminator class to be a large neural network of size growing
with $n, \beta$ and giving a universal approximation of some non-parametric discriminator class dependent on $\beta$. Given
such a fine approximation to discriminator class, one obtains
IPM error bounds (on e.g. $W_1$), which is  then converted to $\TV$ bound via
comparison inequalities. In this paper, we stress again, the discriminator class is extremely
small and weak (a collection of half-spaces) and does not depend on
smoothness $\beta$ or number of samples $n$. We provide more detailed literature review and
discussion in \Cref{appendix:related_works}. 

Another paper by \cite{oko2023diffusion} derives minimax density estimation guarantees for a class of diffusion-based estimators. Their result is similar to ours in that it obtains rigorous (near-)optimal guarantees for a method that is a realistic model of what is currently done in practice. However, their analysis also rely crucially on the universal approximation property of neural networks for the score function.

In this present work,  we mainly focus on the fixed and interpretable discriminator class $\{x\mapsto f(x^\top v-b):\|v\|\leq1,b\in\R\}$ for a set of prescribed $f$ and derive comparisons to $\TV$ for smooth distributions, Gaussian mixtures, and discrete distributions. In addition, we prove the (near)-optimality of our results (for smooth densities) and also derive nonparametric estimation rates for the corresponding GAN density estimators. 
Our work is, to the best of our knowledge, closest to the first that satisfies:
    \begin{enumerate}
        \item Uses a vanilla GAN-type approach (\eqref{eqn:intro adversarial}) with ERM (which is close to practice) and obtains total variation rates that are (close to) minimax optimal.
        
        \item Discriminator class is oblivious to the smoothness parameters: parts of our results are also $n$-oblivious, and the ones that depend on $n$ only concern the activation function. Our proposed discriminator class is also \textit{very} simple: we use a composition of a single non-linearity and a linear
	function (i.e. it is a VC class with dimension $d+1$).
    \end{enumerate}
    
The closest works to the above objectives are \cite{belomestny2021rates} and \cite{stephanovitch2023wasserstein}, model 3, both of which derived log-optimal $\TV$ risk exponent $\widetilde O\left(n^\frac{-\beta}{2\beta+d}\right)$ via \eqref{eqn:intro adversarial}. However, they both relied on the universal approximation of neural networks which has to be larger than some function of $\beta$, the smoothness parameter, as well as $n$, the sample size.
Hence another message, in contrast to the above works on discriminator networks, is that on an \emph{extremely} simple class of discriminators parametrized by half-spaces (even with finite VC dimension) without the need for training/optimization, one can get almost optimal density estimation rates by averaging random halfspace distances (\Cref{cor:d_a density est}). Moreover, if \emph{any} generator fits empirical to within $O(1/\sqrt{n})$ of our distance (this rate is oblivious to $\beta$), we get TV with high probability (\Cref{prop:stopping crit}).

Independent of this work, 
recent results by \cite{paik2023maximum} investigate the halfspace
separability of distributions for the setting of two-sample testing. However, their focus was on
the asymptotic power of the test as the number of samples grows to infinity. Our lower bound construction presented in \Cref{sec: lowers} proves that their proposed test is sub-optimal in the minimax setting. See \Cref{sec: negative} for a more detailed discussion. 
\subsection{Notation}\label{sec:notation} The symbols $O,o,\Theta,\Omega,\omega$ follow the
conventional ``big-O" notation, and $\tilde{O}, \tilde o$ hide polylogarithmic factors. We use $\lesssim$, $\gtrsim$ and $\asymp$ throughout our calculations to hide multiplicative constants that are irrelevant (depending on the context). 
Given a vector $x\in\R^d$, we write $\|x\|$ for its Euclidean norm and $\langle x,y\rangle\eqdef
x^\top y$ for the Euclidean inner product of $x,y\in\R^d$. The
Gamma function is denoted by $\Gamma$. We write $\bb B(x,r)\eqdef
\{y\in\R^d:\|x-y\|\leq r\}$, $\bb S^{d-1} \eqdef \{x\in\R^d:\|x\|=1\}$ and $\sigma$ for the
unnormalized surface measure on $\bb S^{d-1}$. The surface area of a unit $(d-1)$-sphere is also
written as $\sigma(\mathbb{S}^{d-1})=2\pi ^{d/2} /\Gamma {\bigl (}{\frac {d}{2}}{\bigr )}$. In particular, if $X$ is a
random vector uniformly distributioned on $\mathbb{S}^{d-1}$ then for any $h$ we have
$$ \EE[h(X)] = 
\frac{1}{\sigma(\mbb{S}^{d-1})}
\int_{\R^d} h(y) \D\sigma(y)\,.$$
The convolution between functions/measures is denoted by $*$.  We write $L^p(\R^d)$ for the space
of (equivalence classes of) functions $\R^d\to\C$ that satisfy $\|f\|_p \eqdef
\left(\int_{\R^d}|f(x)|^p\D x\right)^{1/p} < \infty$. 
The space of all probability distributions on $\R^d$ is denoted as $\cal P(\R^d)$. For a signed
measure $\nu$ we write $\operatorname{supp}(\nu)$ for its support and $M_r(\nu) \eqdef \int
\|x\|^r\D|\nu|(x)$ for its $r$'th absolute moment. Given $\P,\Q \in \cal P(\R^d)$ we write
$\TV(\P,\Q) \eqdef \sup_{A\subseteq \R^d} [\P(A)-\Q(A)]$ for the total variation distance, where
the supremum is over all measurable sets.

Given a function $f\in L^1(\R^d)$, define its Fourier transform as 
\begin{equation*}
    \widehat f (\omega) \eqdef \cal F[f](\omega) \eqdef \int_{\R^d} e^{-i\langle x,\omega\rangle} f(x)\D x. 
\end{equation*}
Given a finite signed measure $\nu$ on $\R^d$, define its Fourier transform as $\cal F[\nu](\omega) \eqdef \int_{\R^d} e^{-i\langle\omega,x\rangle}\D\nu(x)$. We extend the Fourier transform to 
$L^2(\R^d)$ and tempered distributions in the standard manner. Given $f \in L^2(\R^d)$ and $\beta > 0$, define its homogenous Sobolev seminorm of order $(\beta,2)$ as
\begin{equation}\label{eqn: def sobolev norm}
    \|f\|_{\beta,2}^2 \eqdef \int_{\R^d} \|\w\|^{2\beta} |\widehat f(\w)|^2 \D\w. 
\end{equation}
Further, we define two specific classes of functions of interest as follows: $\cal P_S(\beta,d, C) $ is a set of smooth densities while $\cal P_G(d)$ is a set of all Gaussian mixtures with support in the unit ball, formally
\begin{equation*}
\begin{aligned}
   \cal P_S(\beta,d, C) &\eqdef \{\mu\in\cal P(\R^d): \operatorname{supp}(\mu) \subseteq \bb B(0,1),\text{$\mu$ has density $p$ with }\|p\|_{\beta,2} < C \} ,\\
   \cal P_G(d) &\eqdef \{\nu*\cal N(0,I_d): \nu\in\cal P(\R^d),\operatorname{supp}(\nu) \subseteq \bb B(0,1)\}.
\end{aligned}
\end{equation*}
\begin{assumption}\label{assumption:C}
    Throughout the paper we assume that $C$ in the definition of $\cal P_S(\beta,d,C)$ is large enough relative to $\beta$ and $d$, such that $\cal P_S(\beta,d, C/2)$ is non-empty. 
\end{assumption}

\subsection{Structure}

The structure of the paper is as follows. In Section \ref{sec:equiv formulations} we introduce the generalized energy distance, the main object of our study. We show how it relates to the perceptron discrepancy $\overline{d_H}$ and its relaxation $d_H$; we record equivalent formulations of the generalized energy distance, one of which is a novel ``sliced-distance'' form. In \Cref{sec:TV vs energy}, we present our main technical results on comparison inequalities between total variation and the energy distance. In \Cref{sec:estimator} we analyse the density estimator that minimizes the empirical energy distance, and prove \Cref{thm:intro} and \Cref{thm:intro comparison} in \Cref{sec:proof of intro}. In \Cref{sec: negative} we show that the use of $\overline{d_H}$ for two sample testing results in suboptimal performance. We conclude in \Cref{sec:conclusion}. All omitted proofs and auxiliary results are deferred to the Appendix.

\section{The Generalized Energy Distance}\label{sec:equiv formulations}

Given two probability distributions $\mu,\nu$ on $\R^d$ with finite $\gamma$'th moment, the generalized energy distance of order $\gamma \in (0,2)$ between them is defined as 
\begin{equation}\label{eqn:gen energy def}
    \cal E_\gamma(\mu,\nu) = \E\Big[2\|X-Y\|^\gamma-\|X-X'\|^\gamma-\|Y-Y'\|^\gamma\Big], \qquad\text{where } (X,X',Y,Y') \sim \mu^{\otimes 2}\otimes\nu^{\otimes 2}. 
\end{equation}
As we alluded to in the introduction, the proof of \Cref{thm:intro,thm:intro comparison} becomes possible once we relax the supremum in the definition of $\overline{d_H}$ to an \textit{unnormalized} average over halfspaces. In \Cref{sec:average halfspaces} we discuss this relaxation in more detail and identify a connection to the energy distance $\cal E_1$ defined above in \eqref{eqn:gen energy def}. Motivated by this, we study the (generalized) energy distance and give multiple equivalent characterizations of it from \Cref{sec:average halfspaces} to \Cref{sec:riesz}. 

\subsection{From Perceptron Discrepancy to Energy Distance}\label{sec:average halfspaces}

Our first goal is to connect the study of $\overline{d_H}$ to the study of $\cal E_\gamma$ with
$\gamma=1$. To achieve this, we introduce an intermediary, the ``average'' perceptron discrepancy $d_H$.
Given two probability distributions $\mu,\nu$ on $\R^d$, we define
\begin{equation}\label{eq:dh}
	d_H(\mu,\nu) \eqdef \sqrt{\int_{v\in\mbb{S}^{d-1}}\int_{b\in\R} \left(\int_{\langle v,x\rangle\geq b} \D\mu(x)-\D\nu(x)\right)^2 \d{b}\d\sigma(v)}, 
\end{equation}
where $\sigma$ denotes the surface area measure. 

If the two distributions $\mu,\nu$ are
supported on a compact set, then the overall definition can indeed be regarded as a `mean squared' version of perceptron discrepancy, because the integrals over $b$ and $v$ only range over bounded sets. However, in general, the integral over $b$ in the definition of $d_H$ is not normalizable
and that is why we put ``average'' in quotes. Nevertheless, we have the following comparisons
between $d_H$ and $\overline{d_H}$.

\begin{proposition}\label{prop:d_H < d_H}
    For any $\beta > 0$, $d\geq1$, $C>0$, and for all  $\mu,\nu\in\cal P_S(\beta,d, \infty)$, we have
    \begin{equation}\label{eqn:smooth d_H<d_H}
        \sqrt{\frac{\Gamma(d/2)}{4\pi^{d/2}}} d_H(\mu,\nu) \leq \overline{d_H}(\mu,\nu).
    \end{equation}
    Moreover, for all $d\geq1$, there exists a finite constant $C_1 = C_1(d)$ such that for all $\mu,\nu\in\cal P_G(d)$, 
        \begin{equation*}
            \frac{d_H(\mu,\nu)}{\log(3+1/d_H(\mu,\nu))^{1/4}} \leq C_1\overline{d_H}(\mu,\nu).
        \end{equation*}
    \end{proposition}
\begin{proof}
    The proof of \eqref{eqn:smooth d_H<d_H} is immediate after noting that all distributions in $\cal P_S(\beta,d, \infty)$ are supported on the $d$-dimensional unit ball and that $\int_{v\in\bb S^{d-1}}\int_{-1}^1\D b\D\sigma(v)=4\pi^{d/2}/\Gamma(d/2)$. Thus, we focus on the Gaussian mixture case. Write $\mu-\nu = \tau * \phi$ where $\phi$ denotes the density of the standard Gaussian $\cal N(0,I_d)$ and $\tau\in\mc{P}(\R^d)$ is the difference of the two implicit mixing measures. For any $R>0$, we have
\begin{align*}
    \overline{d_H}(\mu,\nu) &\geq \sup_{v \in \bb S^{d-1}, |b|\leq R} \int_{\langle x,v\rangle\geq b} (\tau*\phi)(x) \D x \\
    &\geq \sqrt{\frac{1}{2R \vol_{d-1}(\bb S^{d-1})} \int_{\bb S^{d-1}} \int_{|b|\leq R} \left( \int_{\langle x,v\rangle\geq b} (\tau*\phi)(x)\D x\right)^2 \D b\D\sigma(v)}.
\end{align*}
Now, since $\tau$ is supported on a subset of $\bb B(0,1)$ by definition of the class $\cal P_G(d)$, for any $v \in \bb S^{d-1}$ and $R\geq2$ we have the bound
\begin{align*}
    \int_{|b| > R} \left(\int_{\langle x,v\rangle\geq b} \int_{\R^d} \phi(x-y)\D \tau(y) \D x\right)^2 \D b &\leq \int_{|b|>R} \left(\int_{\langle x, v\rangle \geq|b|} \exp(-(\|x\|-1)^2/2)\D x\right)^2\D b \\
    &\leq \int_{|b|>R} \left(\int_{\|x\|\geq|b|} \exp(-\|x\|^2/8)\D x\right)^2\D b \\
    &\lesssim \exp(-\Omega(R^2)), 
\end{align*}
where we implicitly used that $\int \D \tau=0$ as $\tau$ is the difference of two probability distributions. Choosing $R \asymp \sqrt{\log(3 + 1/d_H(\mu,\nu))}$ concludes the proof. 
\end{proof}

\Cref{prop:d_H < d_H} implies that to obtain a comparison between $\TV$ and $\overline{d_H}$, specifically for lower bounding $\overline{d_H}$, it suffices to consider the relaxation $d_H$ instead.
 The next observation we make is that $d_H$ is in fact equal, up to constant, to the energy distance. 
\begin{theorem}\label{prop:dh_energy}
    Let $\mu,\nu$ be probability distributions on $\R^d$ with finite mean. Then 
    \begin{equation*}
        d_H(\mu,\nu) = \frac{\pi^{(d-1)/4}}{\sqrt{\Gamma(\frac{d+1}{2}})} \cal E_1(\mu,\nu). 
    \end{equation*}
\end{theorem}

\begin{proof}
    This is a direct implication of \eqref{eqn:norm equiv form} and \eqref{eqn:slice equiv form}. We defer the proof to \Cref{prop:slice equiv form} which is its direct generalization as the interpretation of $\mathcal{E}_\gamma$ as the ``average'' $\mathcal{D}_\gamma$ distance for all $\gamma\in(0, 2)$.
\end{proof}

As will be clear from the rest of the paper, it does pay off to study $\cal E_\gamma$ for general
$\gamma$, even though so far we only justified its relevance to the results stated in the introduction
for the case of $\gamma=1$. With this in mind, we proceed to study various properties of the \textit{generalized} energy distances $\{\cal E_\gamma\}_{\gamma\in(0,2)}$.

\subsection{The Fourier Form}\label{sec:fourier}

The formulation of the generalized energy distance that we rely on most heavily in our proofs is the following. 

\begin{proposition}[{{\cite[Proposition 2]{szekely2013energy}}}]\label{prop:norm equiv form}
    Let $\gamma\in(0,2)$ and let $\mu,\nu$ be probability distributions on $\R^d$ with finite $\gamma$'th moment. Then, 
    \begin{equation}\label{eqn:norm equiv form} 
        \cal{E}_{\gamma}^2(\mu,\nu) = F_\gamma(d) \int_{\R^d} \frac{|\widehat\mu(\omega)-\widehat\nu(\omega)|^2}{\|\w\|^{d+\gamma}} \D\omega, 
    \end{equation}
    where we define $F_\gamma(d)=\frac{\gamma2^{\gamma-1}\Gamma\left(\frac{d+\gamma}{2}\right)}{\pi^{d/2}\Gamma(1-\frac{\gamma}{2})}$. 
\end{proposition}
\begin{remark}
    Note that $F_\gamma(d) = \Theta\left(\gamma(2-\gamma)\Gamma\left(\frac{d+\gamma}{2}\right)\pi^{-d/2}\right)$ up to a universal constant. 
\end{remark}
This shows that the generalized energy distance is a weighted $L^2$ distance in Fourier space. The fact that $\cal E_\gamma$ is a valid metric on probability distributions with finite $\gamma$'th moment is a simple consequence of \Cref{prop:norm equiv form}.

\subsection{The MMD and IPM Forms}\label{sec:mmd form}

Another interpretation of the generalized energy distance is through the theory of \emph{Maximum Mean
Discrepancy} (MMD). Given a set $\cal X$ and a positive semidefinite kernel $k:\cal X^2 \to \R$, there is a unique reproducing kernel Hilbert space (RKHS) $\cal H_k$ consisting of the closure of the linear span of $\{k(x,\cdot),x\in\mc{X}\}$ with respect to the inner product
$ \left\langle k(x,\cdot),k(y,\cdot)\right\rangle_{\mc H_k}=k(x,y).$  

For a probability distribution $\mu$ on $\cal X$, define its kernel embedding as $\theta_{\mu} =
\int_{\R^d} k(x,\cdot)\d\mu(x)$. As shown in \cite[Lemma 3.1]{muandet2017kernel}, the kernel embedding $\theta_{\mu}$ exists and belongs to the RKHS $\mc{H}_k$ if $\EE[\sqrt{k(X,X')}] < \infty$ for $(X,X') \sim \mu^{\otimes 2}$ — as is the case for our kernel defined later in \Cref{eq:kernel_def}.
Then, given two probability distributions $\mu$ and $\nu$, the MMD measures their distance in the RKHS by 
$$\mmd_k(\mu,\nu) \eqdef \|\theta_\mu-\theta_\nu\|_{\cal H_k}.$$ 
We refer the reader to
\cite{10.7551/mitpress/4175.001.0001,muandet2017kernel} for more details on the underlying theory. MMD has a closed form thanks to the reproducing property:
\begin{align*}
    \mmd_k^2(P,Q) &= \E\Big[k(X,X')+k(Y,Y')-2k(X,Y)\Big],
\end{align*}
where $(X,X',Y,Y') \sim \mu^2\otimes\nu^2$. Moreover, it also follows that MMD is an  \emph{Integral Probability Metric (IPM)}  where the supremum is over the unit ball of the RKHS $\cal H_K$:
$$\mmd_k(\mu,\nu)=\sup_{f\in\cal H_k:\|f\|_{\cal H_k} \leq
1} \E[f(X)-f(Y)].$$
In our case, we can define the kernel
\begin{align}\label{eq:kernel_def}
    k_\gamma(x,y)=\|x\|^{\gamma}+\|y\|^{\gamma}-\|x-y\|^{\gamma}
\end{align}
for $\gamma\in(0,2)$, which in one dimension corresponds to the
covariance operator of fractional Brownian motion. For a proof of the nontrivial fact that $k_\gamma$ above is positive definite see for example \cite{sejdinovic2013equivalence}. With the choice of $k_\gamma$ it follows trivially from its definition that the generalized energy distance $\cal E_\gamma$ is equal to the MMD with kernel $k_\gamma$, i.e.
$$\cal E_\gamma(\mu,\nu) = \mmd_{k_\gamma}(\mu,\nu)$$for all distributions $\mu,\nu$ with finite $\gamma$'th moment. It is noteworthy that while $\overline{d_H}$ is by definition an IPM, so is its averaged version $d_H$.

A straightforward consequence of the above characterization is the fact that $\cal E_\gamma$ decays at the parametric rate between empirical and population measures. This is not terribly surprising as analogous results hold for arbitrary MMDs with bounded kernel, see for example \cite[Theorem 7]{gretton2012kernel}. Recall that $M_t(\nu)$ denotes the $t$'th absolute moment of the measure $\nu$.

\begin{lemma}\label{lem:d_a concentration}
    Let $\nu$ be a probability distribution on $\R^d$ and let $\nu_n = \frac1n\sum_{i=1}^n\delta_{X_i}$ for an i.i.d. sample $X_1,\dots,X_n$ from $\nu$. Then, for any $\gamma \in (0,2)$, 
    \begin{equation*}
        \E \cal E_\gamma^2(\nu,\nu_n) \leq \frac{2 M_\gamma(\nu)}{n}.
    \end{equation*}
\end{lemma}
For a high-probability bound when $\nu$ is compactly supported, see \Cref{lem:subgauss concentration}.

\begin{proof} 
Let $\tilde X_1,\dots,\tilde X_n$ be an additional i.i.d. sample from $\nu$, and write $\tilde\nu_n$ for the corresponding empirical measure. Using the definition of $\cal E_\gamma$ in \eqref{eqn:gen energy def}, we can compute
\begin{align*}
    \E \cal E_\gamma^2(\nu_n, \tilde\nu_n) &= \E\Big[\frac{2}{n^2} \sum_{i=1}^n\sum_{j=1}^n \|X_i-\tilde X_j\|^\gamma - \frac{1}{n^2}\sum_{i=1}^n\sum_{j=1}^n \|\tilde X_i-\tilde X_j\|^\gamma \\&\qquad\qquad - \frac{1}{n^2}\sum_{i=1}^n\sum_{j=1}^n \|X_i- X_j\|^\gamma\Big] \\
    &= \frac2n\E\|X_1-X_2\|^\gamma.  
\end{align*}
The conclusion then follows from taking the expectation of the expression $$\E\Big[\cal E^2_\gamma(\tilde\nu_n,\nu_n) \Big|\nu_n\Big] = \cal E^2_\gamma(\nu,\nu_n) + \frac1n\E\|X_1-X_2\|^\gamma$$ and the inequality $|x+y|^\gamma \leq 2^{\max\{0,\gamma-1\}}(|x|^\gamma+|y|^\gamma)$ for all $x,y\in\R$. 
\end{proof}

As a clarification remark, while IPM distances also appear in the formulation of GAN \eqref{eqn:intro adversarial} with respect to the discriminator class $\mathcal{D}$, our distance $\mathcal{E}_\gamma$ is not the IPM of $\mathcal{D}_\gamma$, but rather that of the RKHS ball of $k_\gamma$ defined above.  For $\gamma=1$, the IPM for $\mathcal{D}_\gamma$ returns $\overline{d_H}$, which is lower bounded by (up to constant, see \Cref{prop:d_H < d_H}) $d_H =\Theta(\mathcal{E}_1)$. For $\gamma\neq 1$, we leave the study of the IPM for $\mathcal{D}_\gamma$ as well as the characterization of the RKHS ball of $k_\gamma$ out of scope. Despite not being precisely the IPM of $\mathcal{D}_\gamma$, our next characterization relates $\mathcal{E}_\gamma$ to $\mathcal{D}_\gamma$ in the same way $d_H$ is to $\mathcal{D}_1$ via a ``slicing'' perspective.

\subsection{The Sliced Form}
Another equivalent characterization of the generalized energy distance is in the form of a \textit{sliced} distance. Sliced distances are calculated by first choosing a random direction on the unit sphere, and then computing a one-dimensional distance in the chosen direction between the projections of the two input distributions. For $\gamma \in (0,2)$ define the function
\begin{equation}\label{eqn:psi_gamma def}
\psi_\gamma(x) = \begin{cases}
    |x|^{(\gamma-1)/2} &\text{for } \gamma\neq1 \\ \one\{x\geq0\} &\text{otherwise.}
\end{cases}
\end{equation}
The following result, to the best of our knowledge, has not appeared in prior literature except for the case of $\gamma = 1$. This is also the direct generalization of \Cref{prop:dh_energy}.
\begin{theorem}\label{prop:slice equiv form}
    Let $\gamma\in(0,2)$ and let $\mu,\nu$ be probability distributions on $\R^d$ with finite $\gamma$'th moment. Then
    for $(X, Y)\sim \mu\otimes \nu$ we have \begin{equation}\label{eqn:slice equiv form}
        \cal E^2_\gamma(\mu,\nu) = \frac{1}{S_\gamma}\int_{\bb S^{d-1}}\int_\R \Big[\E\psi_\gamma(\langle X,v\rangle-b)-\E\psi_\gamma(\langle Y,v\rangle-b)\Big]^2\D b\D\sigma(v), 
    \end{equation}
    where $S_\gamma=
    \frac{\pi^{\frac{d}{2}+1} \Gamma(1-\frac{\gamma}{2})}
    {\gamma 2^{\gamma-1} \Gamma(\frac{d+\gamma}{2}) \cos^2(\frac{\pi(\gamma-1)}{4})\Gamma(\frac{1-\gamma}{2})^2}
    $ when $\gamma\neq1$ and $S_1=\frac{\pi^{\frac{d-1}{2}}}{\Gamma(\frac{d+1}{2})}$. 
\end{theorem}
The proof of \Cref{prop:slice equiv form} hinges on computing the Fourier transform of the function $\psi_\gamma$, which can be interpreted as a tempered distribution. We point out a special property of the integral on the right hand side of \eqref{eqn:slice equiv form}. After expanding the square, one finds that the individual terms in the sum are not absolutely integrable for $\gamma\neq1$. However, due to cancellations within the squared quantity, the integral is finite. 

As claimed, using the language of \cite{nadjahi2020statistical}, \Cref{prop:slice equiv form} allows us to interpret $\cal E_\gamma$ as a \emph{sliced} probability divergence. Given $v\in\bb S^{d-1}$, write $\theta_v = \langle v,\cdot\rangle$ and $\theta_v\#\nu=\nu\circ\theta_v$ for the pushforward of $\nu$ under $\theta_v$. We have
\begin{equation*}
    S_\gamma(d)  \cal E^2_\gamma(\mu,\nu) = S_\gamma(1)\int_{\bb S^{d-1}} \cal E^2_\gamma(\theta_v\#\mu,\theta_v\#\nu) \D\sigma(v). 
\end{equation*}
We may also observe that the energy distance $\cal E_1$ is equal to the sliced Cramér-$2$ distance up to constant, which has been studied recently by both theoretical and empirical works \citep{tabor2018cramer,9746016}.\footnote{The Cramér-$p$ distance is simply the $L^p$ distance between cumulative distribution functions. }

\subsection{The Riesz Potential Form}\label{sec:riesz}
The generalized energy distance can also be linked to the Riesz potential \citep[Chapter 1.1]{landkof1972foundations}, which is the inverse of the fractional Laplace operator. Given $0 < s < d$, the Riesz potential $I_s f$ of a compactly supported signed measure $f$ on $\R^d$ is defined (in a weak sense) by
\begin{align*}
    {\displaystyle I_{s}f=f*K_{s}}, 
\end{align*}
where 
$K_{s}(x)=c_s^{-1}\|x\|^{s-d}$ and $c_{s}=\pi ^{{d/2}}2^{s} \Gamma (s/2) \Gamma((d-s)/2)^{-1}$. The Fourier transform of the Riesz kernel is given by
$ \widehat {K_s}(\w)= \|\w\|^{{-s}}$, interpreted as a tempered distribution.
The following proposition is derived by setting $s=\frac{d+\gamma}{2}$ and using the Fourier form (\Cref{prop:norm equiv form}) of the energy distance.
\begin{proposition}\label{prop:riesz}
    Let $\gamma\in(0,\min\{d,2\})$ and let $\mu,\nu$ be compactly supported probability distributions on $\R^d$. Then
    \begin{align}\label{eq:E=Riesz}
        \cal E_\gamma(\mu,\nu) = (2\pi)^{d/2} \sqrt{F_\gamma(d)} \| I_{\frac{d+\gamma}{2}} (\mu-\nu) \|_2.
    \end{align}
\end{proposition}

\section{Main Comparison: TV Versus Energy}\label{sec:TV vs energy}

After considering the connection of the perceptron discrepancy $\overline{d_H}$ to the energy distance in \Cref{sec:equiv formulations}, we turn to some of our main technical results, which provide novel quantitative comparisons between $\{\cal E_\gamma\}_{\gamma\in(0,2)}$ and the total variation distance. In \Cref{sec:riesz upper bd} we show that the generalized energy distance is upper bounded by total variation for compactly supported distributions. In \Cref{sec:TV energy comparison absolutely cont} we derive lower bounds on the generalized energy distance in terms of the total variation distance over the two distribution classes that we have introduced, namely smooth distributions and Gaussian mixtures. Finally, in \Cref{sec:TV energy discrete} we turn to the case of discrete distributions, which requires alternative techniques.

\subsection{Upper Bound --- Arbitrary Compactly Supported Distributions}\label{sec:riesz upper bd}
Note that we (obviously) always have $\overline{d_H}(\mu, \nu) \le \TV(\mu,\nu)$ for arbitrary probability measures $\mu$ and $\nu$. Moreover, for distributions supported on a unit ball we also have $d_H(\mu, \nu) \lesssim
\overline{d_H}(\mu,\nu)$. Therefore, by the identification of $d_H$
and $\cal E_1$ (\Cref{prop:dh_energy}), we can see that
for distributions with bounded support, we always have 
$ \cal E_1(\mu,\nu) \lesssim \TV(\mu,\nu)\,.$
The next result generalizes this estimate for all $\cal E_\gamma$, not just $\gamma=1$.

\begin{proposition} \label{prop:HLS}
For any dimension $d\geq1$ and $\gamma \in (0,2)$ there exists a finite constant $c = c(d, \gamma)$ such that for any two probability distributions $\mu,\nu$ supported on the unit ball we have
	$$ \cal E_\gamma(\mu,\nu) \le c \TV(\mu, \nu)\,.$$
\end{proposition}
\begin{proof}
The proof directly follows from $\|x-y\|^\gamma\leq 2^\gamma$ and so 
$\cal{E}_\gamma^2(\mu,\nu)=\E\Big[2\|X-Y\|^\gamma-\|X-X'\|^\gamma-\|Y-Y'\|^\gamma\Big]=\int -\|x-y\|^\gamma (\d\mu(x)-\d\nu(x))(\d\mu(y)-\d\nu(y))\leq 2^\gamma \TV(\mu,\nu)^2$, where $(X,X',Y,Y') \sim \mu^{\otimes 2}\otimes\nu^{\otimes 2}.$ \end{proof}

\subsection{Lower Bound --- Smooth Distributions And Gaussian Mixtures}\label{sec:TV energy comparison absolutely cont}
In \Cref{sec:riesz upper bd} we showed that the energy distance is upper bounded by total variation for compactly supported measures. In this section we look at the reverse direction, namely, we aim to lower bound the energy distance by total variation.

\begin{theorem} \label{prop:d_a comparison}
    For any $\beta > 0$, $d\geq1$, there exists a finite constant $C_1 = C_1(\beta, d)$ so that  
        \begin{equation}\label{eqn:smooth comparison}
           \sqrt{\gamma(2-\gamma)}\TV(\mu,\nu)^{\frac{2\beta+d+\gamma}{2\beta}} \leq C_1 \cdot (\|\mu\|^2_{\beta, 2}+\|\nu\|^2_{\beta, 2})^{\frac{d+\gamma}{4\beta}}\cdot\cal E_\gamma(\mu,\nu)
        \end{equation}
        for any $\mu,\nu \in \cal P_S(\beta,d, \infty)$ and $\gamma \in (0,2)$. Similarly, for any $d\geq1$ there exists a finite constant $C_2 = C_2(d)$ such that 
        \begin{equation}\label{eqn:gauss comparison}
            \frac{\sqrt{\gamma(2-\gamma)}\TV(\mu,\nu)}{\log(3+1/\TV(\mu,\nu))^{\frac{2d+\gamma}{4}}} \leq C_2 \cal E_\gamma(\mu,\nu)
        \end{equation}
        for every $\mu,\nu\in\cal P_G(d)$ and $\gamma\in(0,2)$.
\end{theorem}

\begin{proof}
    Abusing notation, identify $\mu$ and $\nu$ with their Lebesgue densities.
    The argument proceeds trough a chain of inequalities:
    \begin{enumerate}
        \item Bound $\TV$ by the $L^2$ distance between densities. 
        \item Apply Parseval's Theorem to pass to Fourier space. 
        \item Apply H\"older's inequality with well-chosen exponents. 
    \end{enumerate}
    \textbf{Proof of \eqref{eqn:smooth comparison}}. Jensen's inequality implies that $$2\TV(\mu,\nu) = \|\mu-\nu\|_1 \leq \sqrt{\vol(\bb B(0,1))}\|\mu-\nu\|_2 \lesssim \|\mu-\nu\|_2,$$where $\vol$ denotes volume and we discard dimension-dependent constants. This completes the first step of our proof. For the second step note that $\mu,\nu \in L^2(\R^d)$ and we may apply Parseval's theorem to obtain
\begin{equation*}
    \|\mu-\nu\|_2^2 = \frac{1}{(2\pi)^d}\|\widehat \mu-\widehat \nu\|_2^2. 
\end{equation*}
For arbitrary $\varphi > 0$ and $r\in[1,\infty]$, H\"older's inequality with exponents $\frac1r+\frac{1}{r^*}=1$ implies that 
\begin{equation}\label{eq:holder}
\begin{aligned}
    \|\widehat \mu-\widehat \nu\|_2^2 &= \int_{\R^d} |\widehat\mu(\w)-\widehat\mu(\w)|^2 \frac{\|\w\|^\varphi}{\|\w\|^\varphi} \D\w \\
    &\leq \left(\int_{\R^d} |\widehat \mu(\w)-\widehat \nu(\w)|^2 \|\w\|^{\varphi r}\D\w\right)^{1/r} \left(\int_{\R^d} \frac{|\widehat \mu(\w)-\widehat \nu(\w)|^2}{\|\omega\|^{\varphi r^*}} \D\w\right)^{1/r^*}.
\end{aligned}
\end{equation}
Now, we choose $\varphi$ and $r$ to satisfy
\begin{align*}
    \varphi r &= 2\beta \\
    \varphi r^* &= d+\gamma.
\end{align*}
The first equation ensures that the first integral term is bounded by $\|\mu-\nu\|_{\beta,2}^{2/r}$, which is assumed to be at most a $d,\beta$ dependent constant. The second equation ensures that the second integral term is equal to $(\cal E_\gamma(\mu,\nu)^2 / F_\gamma(d))^{1/r^*}$ by \Cref{prop:norm equiv form}. The solution to this system of equations is given by $r^*=(2\beta+d+\gamma)/(2\beta)$ and $\varphi = 2\beta\cdot \frac{d+\gamma}{2\beta+d+\gamma}$. Note that clearly $\varphi > 0$ and $r^* \geq 1$. Thus, after rearrangement and using that $F_d(\gamma)=\Theta(\gamma(2-\gamma))$ up to a dimension dependent constant, we obtain
\begin{align*}
    \sqrt{\gamma(2-\gamma)}\|\widehat \mu-\widehat \nu\|_2^{\frac{2\beta+d+\gamma}{2\beta}} \leq C_1 \cdot (\|\mu\|^2_{\beta, 2}+\|\nu\|^2_{\beta, 2})^{\frac{d+\gamma}{4\beta}}\cdot \cal E_\gamma(\mu,\nu),  
\end{align*}
for a finite constant $C_1 = C_1(d,\beta)$, concluding the proof. 

\textbf{Proof of \eqref{eqn:gauss comparison}}. We write $C(d) \in (0,\infty)$ for a dimension dependent constant that may change from line to line. The outline of the argument is analogous to the above, with the additional step of having to bound the $(\beta,2)$-Sobolev norm of the Gaussian density as $\beta \to \infty$ for which we rely on \Cref{lem: Dbeta of GMM}. Let $\mu$ and $\nu$ have densities $p*\phi$ and $q*\phi$, where $\phi$ is the density of $\cal N(0,I_d)$. Writing $f=(p-q)$, we can extend 
the proof of \cite[Theorem 22]{jia2023entropic} to multiple dimensions to find, for any $R > 2$, that 
    \begin{align*}
        2\TV(\mu,\nu) = \|\mu-\nu\|_1 &= \int_{\|x\| \leq R}\left| (f * \phi)(x)\right| \D x + \int_{\|x|| > R}\left|\int_{\R^d} \phi(x-y)\D f(y)\right| \D x \\
        &\leq \sqrt{\vol_d(\bb B(0,R))} \sqrt{\int_{\|x\|\leq R} |(f*\phi)(x)|^2\D x} + \int_{\|x\|> R}\exp(-\|x\|^2/8)\D x \\
        &\leq C(d)\Big(R^{d/2} \|\mu-\nu\|_2 + \exp(-\Omega(R^2))\Big), 
    \end{align*}
    where the second line uses that $\supp(f)\subseteq\bb B(0,1)$. Taking $R \asymp \sqrt{\log(3 + 1/\|\mu-\nu\|_2)}$ we obtain the inequality
    \begin{equation}\label{eqn:gauss tv L2}
        \TV(\mu,\nu) \leq C(d)\|\mu-\nu\|_2\log(3+1/\|\mu-\nu\|_2)^{d/4}. 
    \end{equation}
    By H\"older's inequality we obtain
\begin{align*}
    \|\widehat f\|_2 &\leq \|\|\omega\|^{\beta}\widehat f(\omega)\|_2^{\frac{d+\gamma}{2\beta+d+\gamma}} \left\|\frac{\widehat f(\omega)}{\|\omega\|^{\frac{d+\gamma}{2}}}\right\|_2^{\frac{2\beta}{2\beta+d+\gamma}}
    \\ 
    &= \|\|\omega\|^{\beta}\widehat f(\omega)\|_2^{^{\frac{d+\gamma}{2\beta+d+\gamma}}}\cdot \cal E_\gamma(\mu, \nu)^{{\frac{2\beta}{2\beta+d+\gamma}}} \cdot F_\gamma(d)^{-\frac{\beta}{2\beta+d+\gamma}}
\end{align*}
by \Cref{prop:norm equiv form}. Using that $|\widehat f| \leq |\widehat \phi|$ and applying \Cref{lem: Dbeta of GMM}, for $\beta \geq 1$ we get
\begin{align*}
    F_\gamma(d)^{\frac{\beta}{2\beta+d+\gamma}}\|\widehat f\|_2 \leq \cal E_\gamma(\mu, \nu)^{{\frac{2\beta}{2\beta+d+\gamma}}} \left(\frac{5\pi^{d/2}}{\Gamma(d/2)} \left(\frac{2\beta+d}{2e}\right)^{\frac{2\beta+d-1}{2}}\right)^{\frac{d+\gamma}{2(2\beta+d+\gamma)}}. 
\end{align*}
Rearranging and using Parseval's Theorem, we get
\begin{align*}
    \cal E_\gamma(\mu,\nu) \geq C(d) \sqrt{\gamma(2-\gamma)} \|f\|_2 \frac{\|f\|_2^{\frac{d+\gamma}{2\beta}}}{\left(\frac{2\beta+d}{2e}\right)^{\frac{(d+\gamma)(2\beta+d-1)}{8\beta}}}
\end{align*}
for some $d$-dependent, albeit exponential, constant $C(d)>0$. Plugging in $\beta = \log(3+1/\|f\|_2)$ and assuming that $\|f\|_2$ is small enough in terms of $d$, we obtain
\begin{align}\label{eqn:E_g vs L2}
    \cal E_\gamma(\mu,\nu) &\geq \frac{C(d)\sqrt{\gamma(2-\gamma)}\|f\|_2}{\log(3+1/\|f\|_2)^{\frac{d+\gamma}{4}}} \geq \frac{C(d)\sqrt{\gamma(2-\gamma)}\TV(\mu,\nu)}{\log(3+1/\TV(\mu,\nu))^{\frac{2d+\gamma}{4}}},
\end{align}
where the second inequality uses \eqref{eqn:gauss tv L2} and \Cref{lem:log inverse}.
\end{proof}

\Cref{prop:d_a comparison} is our main technical result, which shows that $\cal E_\gamma$ is lower bounded by a polynomial of the total variation distance for both the smooth distribution class $\cal P_S$ and Gaussian mixtures $\cal P_G$. Note also that in one dimension, \eqref{eqn:smooth comparison} follows from the Gagliardo–Nirenberg-Sobolev interpolation inequality. However, to our knowledge, the inequality is new for $d>1$. As for the tightness of \Cref{prop:d_a comparison}, we manage to prove that this inequality is the best possible for $\cal P_S$ in one dimension, and best possible up to a poly-logarithmic factor in dimension $2$ and above. 

\begin{proposition}\label{thm:tightness}
For any
$\beta > 0$, $d\geq1$, $\gamma \in (0,2)$ and $C>0$ satisfying \Cref{assumption:C}, there exists a finite constant $C_1 = C_1(\beta, d, \gamma, C)$ so that for any value of $\epsilon \in (0,1)$, there exist $\mu_\epsilon,\nu_\epsilon \in \cal P_S(\beta,d, C)$ such that $  \TV(\mu_\epsilon,\nu_\epsilon)/\epsilon \in (1/C_1,C_1)$ and 
        \begin{equation*}
            \cal E_\gamma(\mu_\epsilon, \nu_\epsilon)  \leq C_1 \TV(\mu_\epsilon,\nu_\epsilon)^{\frac{2\beta+d+\gamma}{2\beta}} \log\left(3 + \frac{1}{\TV(\mu_\epsilon,\nu_\epsilon)}\right)^{d-1}.
        \end{equation*}
\end{proposition}

In the special case $\gamma=1$ we obtain an even stronger notion of tightness. 

\begin{proposition}\label{prop:d_H tightness}
When $\gamma=1$ we may replace $\cal E_1$ by $\overline{d_H}$ in \Cref{thm:tightness}.
\end{proposition}

\Cref{prop:d_H tightness} is an improvement over \Cref{thm:tightness} due to the inequality $d_H
\lesssim \overline{d_H}$ over the class $\cal P_S(\beta,d, C)$, which follows from \Cref{prop:d_H
< d_H}. It shows also that our construction has the property that there does not exist any
halfspace that separates $\mu$ and $\nu$ better than our bounds suggest. 

The proofs of both results are presented in \Cref{sec: lowers}. The general idea is to saturate H\"older's inequality in~\eqref{eq:holder}, for which the Fourier transform of $f =
\frac{\D\nu}{\D x} - \frac{\D\mu}{\D x}$ should be supported on a sphere. However, such $f$
clearly cannot be compactly supported. Thus the actual
construction is to multiply the Fourier inverse of the uniform measure on a sphere with a compactly
supported mollifier. In $d>1$ the mollifier that we require must have super-polynomial Fourier
spectrum decay, for which we use the recent construction in~\cite{cohen2023fractal}.

\subsection{Lower Bound --- Discrete Distributions}\label{sec:TV energy discrete}

Suppose we have two discrete distributions that are supported on a common, finite set of size $k$. One way to measure the energy distance between them would be to identify their support with the set $\{1,2,\dots,k\}$, thereby embedding the two distributions in $\R$, and applying the one-dimensional energy distance. 

While the above approach seems reasonable, it is entirely arbitrary. Indeed, there might not be a natural ordering of the support; moreover, why should one choose the integers between $1$ and $k$ instead of, say, the set $\{1,2,4,\dots,2^k\}$? The total variation distance does not suffer from such ambiguities, and it is unclear how our choice of embedding affects the relationship to $\TV$. The following result attacks precisely this question. 

\begin{theorem}\label{prop:discrete comparison}
    Let $\mu$ and $\nu$ be probability distributions supported on the set $\{x_1,\dots,x_k\}\subseteq \R^d$ and let $\delta=\min_{i\neq j}\|x_i-x_j\|$. Then there exists a universal constant $C>0$ such that 
    \begin{equation*}
        \cal E_1^2(\mu,\nu) \geq \frac{C\delta}{k\sqrt d}\TV^2(\mu,\nu). 
    \end{equation*}
\end{theorem}
\begin{proof}
    Let $\mu = \sum_{i=1}^k \mu_i\delta_{x_i}$ and $\nu = \sum_{i=1}^k\nu_i\delta_{x_i}$. Then, by \cite[Theorem 1]{ball1992eigenvalues} we have
    \begin{align*}
        \cal E^2_1(\mu,\nu) &= -\sum_{i,j}(\mu_i-\nu_i)(\mu_j-\nu_j)\|x_i-x_j\| \geq \frac{C\delta}{\sqrt d} \sum_{i=1}^k (\mu_i-\nu_i)^2 \geq \frac{C\delta\TV^2(\mu,\nu)}{k\sqrt d}.
    \end{align*}
    \end{proof}

\begin{remark}
    Similar results can be proved for the generalized energy distance $\cal E_\gamma$, using e.g. the work \cite{narcowich1992norm}. However, to the best of our knowledge, these estimates degrade significantly in the dimension $d$ in contrast with \cite{ball1992eigenvalues}. 
\end{remark}

Notice that by our discussion above, the support set $\{x_1,\dots,x_k\}$  in \Cref{prop:discrete comparison} is arbitrary and may be chosen by us. Since the scale of the supporting points $x_1,\dots,x_k$ is statistically irrelevant, we remove this ambiguity by restricting the points to lie in the unit ball, i.e. requiring that $\max_i\|x_i\|\leq1$. We see now that the comparison between $\cal E_1 $ and $\TV$ \emph{improves} as $\delta/\sqrt d$ grows. Given a fixed value of $\delta$, we want to make the dimension $d$ of our embedding as low as possible, which means that the points $x_1,\dots,x_k$ should form a large $\delta$-packing of the $d$-dimensional unit ball. Due to well known bounds on the packing number of the Euclidean ball, it follows that the best one can hope for is 
\begin{equation*}
\log(k) \asymp d\log(1/\delta). 
\end{equation*}
Maximizing $\delta/\sqrt d$ subject to this constraint yields the choice $d=\Theta(\log(k))$ and $\delta=\Theta(1)$. This gives us the following corollary.

\begin{corollary}\label{cor:discrete comparison}
    There exists a universal constant $C\in(0,\infty)$ such that for any $k\geq1$ there exists a set of points $x_1,\dots,x_k \in \R^{\lceil{C\log(k)}\rceil}$ with $\max_i\|x_i\|\leq1$ such that 
    \begin{equation*}
        \cal E_1\left(\sum_{i=1}^k\mu_i\delta_{x_i}, \sum_{i=1}^k\nu_i\delta_{x_i}\right) \geq \frac{\TV(\mu,\nu)}{C\sqrt k\sqrt[4]{\log(k)}}
    \end{equation*}
    for any two probability mass functions $\mu=(\mu_1,\dots,\mu_k)$ and $\nu=(\nu_1,\dots,\nu_k)$. 
\end{corollary}

The question arises how the set of points $x_1,\dots,x_k$ in \Cref{cor:discrete comparison} should
be constructed. One solution is to use an error correcting code (ECC), whereby we take the $x_i$
to be the codewords of an ECC on the scaled hypercube $\frac{1}{\sqrt{d}} \{\pm1\}^d$ for some
dimension $d$ (known as ``blocklength'' in this context). An ECC is \textit{asymptotically good} if the message length $\log(k)$
is linear in the blocklength $d$, that is $d\asymp\log(k)$, and if the minimum Hamming distance
between any two codewords is $\Theta(d)$, which translates precisely into $\delta = \min_{i\neq
j}\|x_i-x_j\|\asymp1$. Many explicit constructions of asymptotically good error correcting codes
exist, see \cite{justesen1972class} for one such example, and random codes are almost surely
good~\citep{barg2002random}. Clearly the better the code is, the
better the constants we obtain in \Cref{cor:discrete comparison}. 

\begin{remark}
    One interesting consequence of \Cref{cor:discrete comparison} and the preceeding discussion is the following: given a categorical feature with $k$ possible values, the perceptron may obtain better performance by identifying each category with the codewords $x_1,\dots,x_k$ of an ECC instead of the standard one-hot encoding. 
\end{remark}

\section{Density Estimation}\label{sec:estimator}

In this section we apply what we've learnt about the generalized energy distance and the perceptron discrepancy in prior sections, and analyze multiple problems related to density estimation.

\subsection{Estimating Smooth Distributions and Gaussian Mixtures}\label{sec:conve_in_L2}

Suppose that $X_1,\dots,X_n \stackrel{iid}{\sim} \nu$ for some probability distribution $\nu$ on $\R^d$. Given a class of ``generator'' distributions $\cal G$ and $\gamma \in (0,2)$, define the minimum-$\cal E_\gamma$ estimator as
\begin{equation}\label{eqn:min d_a est}
    \tilde\nu_\gamma \in \argmin_{\nu'\in\cal G} \cal E_\gamma(\nu',\nu_n), 
\end{equation}
where $\nu_n = \frac1n\sum_{i=1}^n \delta_{X_i}$. Note that $\tilde\nu_\gamma$ does not quite agree with our definition of $\tilde \nu_\gamma$ in \eqref{eqn:intro adversarial}, because the $\gamma=1$ case minimizes the \textit{average} halfspace distance $d_H\asymp\cal E_1$ and not the perceptron discrepancy $\overline{d_H}$. The following two results bound the performance of $\tilde\nu$ as defined in \eqref{eqn:min d_a est}, as an estimator of $\nu$ for the smooth density class $\cal P_S$ as well as the Gaussian mixture class $\cal P_G$. In \Cref{sec:proof of intro} we present the adapation of these to $\overline{d_H}$, thereby proving \Cref{thm:intro}. 

\begin{theorem}\label{cor:d_a density est}
    Let $\tilde\nu_\gamma$ be the estimator defined in \eqref{eqn:min d_a est}. For any $\beta > 0$, $d\geq1$ and $C>0$, there exists a finite constant $C_1 = C_1(\beta, d, C)$ so that  
    \begin{equation}\label{eqn:energy estimation rate smooth}
        \sup_{\nu\in\cal P_S(\beta,d, C)}\E \TV(\tilde\nu_\gamma,\nu) \leq C_1 (n\gamma(2-\gamma))^{-\frac{\beta}{2\beta+d+\gamma}}
    \end{equation}
    holds for $\cal G=\cal P_S(\beta,d, C)$ and any $\gamma \in (0,2)$. Similarly, for any $d\geq1$ there is a finite constant $C_2 = C_2(d)$ such that 
    \begin{equation}\label{eqn:energy estimation rate gauss}
        \sup_{\nu \in \cal P_G(d)} \E \TV(\tilde\nu_\gamma,\nu) \leq C_2 \phi\left((n\gamma(2-\gamma))^{-1/2}\right)
    \end{equation}
    holds for $\cal G=\cal P_G(d)$ and any $\gamma \in (0,2)$, where $\phi(x) = {x}\cdot {\log(3+1/{x})^{\frac{2d+\gamma}{4}}}$. 
\end{theorem}

\begin{proof}
    Let us focus on the case $\cal G = \cal P_S(\beta,d, C)$ first and let $t = \frac{2\beta+d+\gamma}{2\beta}$. The inequality $\cal E_\gamma(\tilde\nu_\gamma,\nu_n) \leq \cal E_\gamma(\nu,\nu_n)$ holds almost surely by the definition of $\tilde\nu_\gamma$. Writing $C_1=C_1(\beta,d,C)$ for a finite constant that we relabel freely, the first claim is substantiated by the chain of inequalities
    \begin{align*}
        \E \TV(\tilde\nu_\gamma,\nu) &\,\,\stackrel{\text{Thm. \ref{prop:d_a comparison}}}{\leq} \E \Big[\left(C_1\frac{\cal E_\gamma(\tilde\nu_\gamma,\nu)}{\sqrt{\gamma(2-\gamma)}}\right)^{1/t}\Big] \\
        &\stackrel{\Delta-\text{ineq.}}{\leq}  \E \Big[\left(C_1\frac{\cal E_\gamma(\nu,\nu_n) + \cal E_\gamma(\tilde\nu_\gamma,\nu_n)}{\sqrt{\gamma(2-\gamma)}}\right)^{1/t}\Big] \\
        &\stackrel{\text{Eq. \eqref{eqn:min d_a est}}}{\leq} \E \Big[\left(2C_1\frac{\cal E_\gamma(\nu,\nu_n)}{\sqrt{\gamma(2-\gamma)}}\right)^{1/t}\Big] \\
        &\,\,\stackrel{\text{Jensen's}}{\leq} \left(2C_1\frac{\E \cal E_\gamma(\nu,\nu_n)}{\sqrt{\gamma(2-\gamma)}}\right)^{1/t} \\
        &\,\,\stackrel{\text{Lem. \ref{lem:d_a concentration}}}{\leq} \left(\frac{n\gamma(2-\gamma)}{8C_1^2}\right)^{-1/{2t}}. 
    \end{align*}
    The result for $\cal G = \cal P_G$ follows analogously. Define $r(x) = x \cdot \log\left(3+1/x\right)^{-t}$ where $t = \frac{2d+\gamma}{4}>0$. By direct calculation, one can check that $r$ is strictly increasing and convex on $\R_+$. As a consequence, its inverse $r^{-1}$ is strictly increasing and concave. Let $C_2$ be a $d$-dependent finite constant which we relabel repeatedly. Similarly to the case of smooth distributions covered above, we obtain the chain of inequalities as follows:
    \begin{align*}
        \E\TV(\tilde\nu_\gamma,\nu) &\stackrel{\text{Thm.} \ref{prop:d_a comparison}}{\leq} \E\Big[r^{-1}\left(C_2 \frac{\cal E_\gamma(\tilde\nu_\gamma,\nu)}{\sqrt{\gamma(2-\gamma)}}\right)\Big] \\
        &\stackrel{\text{Jensen's}}{\leq}  r^{-1}\left(C_2 \frac{\E \cal E_\gamma(\tilde\nu_\gamma,\nu)}{\sqrt{\gamma(2-\gamma)}}\right) \\
        &\stackrel{\text{Eqn. \ref{eqn:min d_a est}}}{\leq} r^{-1}\left(2C_2\frac{\E \cal E_\gamma(\nu,\nu_n)}{\sqrt{\gamma(2-\gamma)}}\right) \\
        &\stackrel{\text{Lem.} \ref{lem:d_a concentration}}{\leq} r^{-1}\left(C_2(n\gamma(2-\gamma))^{-1/2}\right). 
    \end{align*}
    The conclusion follows by \Cref{lem:log inverse}. 
\end{proof}

Notice that the rate of estimation of the minimum $\cal E_\gamma$ density estimator improves as $\gamma \downarrow 0$, and in fact seems to approach the optimum. However, simultaneously, the ``effective sample size'' $n\gamma$ shrinks. The best trade-off that we can derive by adaptively setting $\gamma$ to be dependent on $n$ is the following. 
\begin{corollary}\label{prop:optimize_alpha}
For any $\beta > 0$, $d\geq1$ and $C>0$, there exists a finite constant $C_1 = C_1(\beta, d, C)$ such that: for all $n\geq 2$, with $\tilde{\nu}_{\gamma_n} \in \arg\min_{\nu^{\prime} \in \mathcal{G}}\mathcal{E}_{\gamma_n}\left(\nu^{\prime}, \nu_n\right)$ in the setting of \eqref{eqn:min d_a est} where $\gamma_n=\log(n)^{-1}$ and $\cal G=\cal P_S(\beta,d, C)$, one has:
    \begin{equation}
        \sup_{\nu\in\cal P_S(\beta,d, C)}\E \TV(\tilde\nu_{\gamma_n},\nu) \leq C_1 \left(\frac{\log n}{n}\right)^{\frac{\beta}{2\beta+d}}.
    \end{equation}
    Similarly, for any $d\geq1$ there is a finite constant $C_2 = C_2(d)$ such that for all $n\geq 3$, with $\tilde{\nu}_{\gamma_n} \in \arg\min_{\nu^{\prime} \in \mathcal{G}}\mathcal{E}_{\gamma_n} \left(\nu^{\prime}, \nu_n\right)$ where $\gamma_n=\log\log(2n)^{-1}$ and $\cal G=\cal P_G(d)$, one has:
    \begin{equation}
        \sup_{\nu \in \cal P_G(d)} \E \TV(\tilde\nu_\gamma,\nu) \leq C_2 \log(n)^{d/2}\sqrt{\log\log n}/\sqrt n.
    \end{equation}
\end{corollary}


\subsection{Proof of \Cref{thm:intro} and \Cref{thm:intro comparison}}\label{sec:proof of intro}

We already have everything needed to deduce \Cref{thm:intro comparison}. Since it is an exercise in combining results, we simply list the required steps:
\begin{enumerate}
    \item Use \Cref{prop:d_a comparison} to get a comparison between $\TV$ and $\cal E_\gamma$.
    \item Set $\gamma=1$ and use \Cref{prop:dh_energy} to get the equivalence between $\cal E_1$ and $d_H$.
    \item Use \Cref{prop:d_H < d_H} to get a comparison between $d_H$ and $\overline{d_H}$. 
\end{enumerate} 

Turning to the proof of \Cref{thm:intro}, we find that it is completely analogous to the proof of \Cref{cor:d_a density est}, with the only difference being that we can no longer rely on \Cref{lem:d_a concentration} to show that the distance between empirical and population measures decays at the parametric rate, as the latter applies to $\cal E_\gamma$ instead of $\overline{d_H}$. However, the corresponding result for $\overline{d_H}$ is well known. 
\begin{lemma}\label{lem:VC bound}
    Let $\nu$ be a probability distribution on $\R^d$ and $\nu_n = \frac1n\sum_{i=1}^n\delta_{X_i}$ for i.i.d. observations $X_i \sim \nu$. Then, for a finite universal constant $C$, 
    \begin{equation*}
        \E \overline{d_H}(\nu,\nu_n) \leq C\sqrt{\frac{d}{n}}. 
    \end{equation*} 
\end{lemma}
\begin{proof}
    Follows for example from \cite[8.3.23]{vershynin2018high} and the fact that $\cal D_a$, the family of halfspace indicators, has VC dimension $d+1$. 
\end{proof}

With \Cref{lem:VC bound} in hand, completing the argument is straightforward: To deduce \Cref{thm:intro} follow the same steps as in the proof of \Cref{cor:d_a density est}, except use \Cref{thm:intro comparison} and \Cref{lem:VC bound} in place of \Cref{prop:d_a comparison} and \Cref{lem:d_a concentration} respectively.

\subsection{Estimating Discrete Distributions}

In many practical machine learning tasks the data is discrete, albeit on a large alphabet $[k]=\{1,2,\dots,k\}$: for example, in recommender systems the alphabet could be all possible ads, products or articles. A common idea to apply modern learning pipelines to such data is to use an embedding $E:[k]\to \R^d$, with ``one-hot” encoding ($d=k$) being the most popular choice. After such an embedding, the data is effectively made “continuous” and the density estimation methods as discussed previously can be applied. Can such an approach be good in the sense of minimax estimation guarantees? We answer this question positively in this section, provided that embedding $E$ comes from an error-correcting code.

Let $\cal P_k$ denote the set of all probability distributions on the set $[k]$. Suppose we observe an i.i.d. sample $X_1,\dots,X_n \sim \nu$ from some unknown distribution $\nu \in \cal P_k$. The problem of estimating $\nu$ is effectively trivial: the empirical distribution provides a minimax optimal estimator. Indeed, it is a folklore fact, see for example \cite[Theorem 1]{canonne2020short} or \cite[Exc. VI.8]{yuryyihongbook}, that the optimal rate of estimation is given by 
\begin{equation}\label{eqn:discrete n_est}
    \sup_{\nu\in\cal P_k} \E\TV^2\left(\frac1n\sum_{i=1}^n\delta_{X_i},\nu\right) \asymp \min\left\{\frac{k}{n}, 1\right\}. 
\end{equation}

Recall from \Cref{sec:TV energy discrete} that we may choose to embed the alphabet $[k]$ into some higher dimensional Euclidean space. Given distinct points $x_1,\dots,x_k\in\R^d$ for some $d\geq1$, we can identify any distribution $\mu \in \cal P_k$ with the probability distribution $\sum_{i=1}^k \mu_i \delta_{x_i}$, where $\mu_i$ is the mass that $\mu$ puts on $i \in [k]$.  

\begin{theorem}\label{thm:discrete estimation}
    There exists a universal constant $C<\infty$ with the following property. For any alphabet size $k$ there exist embedding points $a_1,\dots,a_k \in \R^{\lceil C\log(k) \rceil}$ such that given an i.i.d. sample $X_1,\dots,X_n \sim \nu$ from an unknown $\nu\in\cal P_k$, any estimator $\tilde\nu \in \cal P_k$ that satisfies
    \begin{equation}\label{eqn:discrete estimator def}
        \cal E_1^2\left(\sum_{i=1}^k\tilde\nu_i\delta_{a_i}, \frac1n\sum_{i=1}^n \delta_{a_{X_i}}\right) \leq \frac{c}{n}
    \end{equation}
    for any $c$ enjoys the performance guarantee
    \begin{equation}\label{eqn:discrete estimator rate}
        \sup_{\nu\in\cal P_k} \E \TV^2(\tilde \nu, \nu) \leq C \min\left\{(c+1)\frac{k\sqrt{\log(k)}}{n}, 1\right\}. 
    \end{equation}
    Moreover, we may replace $\cal E_1$ by $\overline{d_H}$ in \eqref{eqn:discrete estimator def} and the result \eqref{eqn:discrete estimator rate} remains true with $\sqrt{\log(k)}$ replaced by $\log(k)$. 
\end{theorem}
\begin{proof}
    Let $a_1,\dots,a_k \in \R^d$ be the points defined in \Cref{cor:discrete comparison} (relabeled from $x_1,\dots,x_k$ for clarity) so that $d\asymp \log(k)$. By the triangle inequality we have
    \begin{align*}
        \E \cal E^2_1\left(\sum_{i=1}^k \tilde\nu_i\delta_{a_i}, \sum_{i=1}^k\nu_i\delta_{a_i}\right) &\leq 2\E \cal E^2_1\left(\sum_{i=1}^k \tilde\nu_i\delta_{a_i}, \frac1n\sum_{i=1}^n\delta_{a_{X_i}}\right) + 2\E \cal E^2_1\left(\frac1n\sum_{i=1}^n \delta_{a_{X_i}}, \sum_{i=1}^k\nu_i\delta_{a_i}\right) \\
        &\mkern-16mu\stackrel{\text{\Cref{lem:d_a concentration}}}{\lesssim} \frac{c+\max_i\|a_i\|}{n} \lesssim \frac{c+1}{n}.
    \end{align*}
    By \Cref{cor:discrete comparison}, the definition of $\tilde\nu$ and the triangle inequality it follows that 
    \begin{equation*}
        \E \TV^2(\tilde\nu,\nu) \lesssim \frac{k \sqrt d(c+1)}{n} \asymp \frac{k\sqrt{\log(k)}(c+1)}{n}.
    \end{equation*}
    Noting the trivial fact that $\TV\leq1$ completes the proof of the first claim. 
    
    Suppose now that we replace $\cal E_1$ by $\overline{d_H}$ in the definition of $\tilde\nu$. The proof follows analogously, using the chain of inequalities
    \begin{equation*}
        \frac{\TV}{\sqrt k\sqrt[4]{d} } \stackrel{\text{Cor. \ref{cor:discrete comparison}}}\lesssim \cal E_1 \stackrel{\text{Prop. \ref{prop:dh_energy}}}{\asymp} \frac{\sqrt{\Gamma\left(\frac{d+1}{2}\right)}}{\pi^{(d-1)/4}} d_H \stackrel{\max_i\|a_i\|\leq1}\lesssim \overline{d_H}, 
    \end{equation*}
    and \Cref{lem:VC bound} in place of \Cref{lem:d_a concentration}, which is where we loose the $\sqrt d\asymp\sqrt{\log(k)}$ factor. 
\end{proof}

As we explained, the empirical distribution $\tilde \nu = \frac1n\sum_{i=1}\delta_{X_i}$ achieves optimality in \eqref{eqn:discrete n_est}, and clearly also achieves $c=0$ in \eqref{eqn:discrete estimator def} i.e. minimizes the empirical risk globally. The point of \Cref{thm:discrete estimation} is to show that approximate minimizers, such as those found via SGD, are also nearly minimax optimal.

\subsubsection{Estimating H\"older Smooth Densities}   
\Cref{thm:discrete estimation} has interesting implications for density estimation over the class of distributions on the cube $[0,1]^d$ with uniformly bounded derivatives up to order $\underline\beta\eqdef \lceil\beta-1\rceil$ and $(\beta-\underline{\beta})$-H\"older continuous $\underline\beta^{th}$ derivative; call such distributions simply $\beta$-H\"older smooth.\footnote{Note that this class is not the same as $\cal P_S$, although related.} Writing $B_j$ for the cube with center $(j-\frac12)\epsilon^{1/\beta}$ and sidelength $\epsilon^{1/\beta}$ where $j\in\{1,\dots,\epsilon^{-1/\beta}\}^d$, it is known that  
    \begin{equation*}
        \sum_j \left|\int_{B_j} (f(x)-g(x))\d x\right| = c\int_{[0,1]^d} |f(x)-g(x)|\d x + O(\epsilon)
    \end{equation*}
    for any $\beta$-H\"older smooth densities $f, g$ and an $\epsilon$-independent constant $c>0$, see for example \cite[Lemma 7.2]{arias2018remember} or \cite[Proposition 2.16]{ingster2003nonparametric}. In other words, discretizing such distributions using a regular grid with $\Omega(\epsilon^{-d/\beta})$ cells maintains total-variation distances up to an additive $O(\epsilon)$ error. 
    
    Now, consider a `multilayer perceptron', that is, a fully connected multilayer neural network with activations given by $x \mapsto \one\{x\geq0\}$. Such a multilayer network with large enough hidden layers can in principle implement the discretization described above, and embed the $\epsilon^{-d/\beta}$ cells as an error correcting code. Thus, due to \Cref{thm:discrete estimation}, the ERM density estimator \eqref{eqn:intro adversarial} would achieve the minimax optimal density estimation rate $n^{-\beta/(2\beta+d)}$ over $\beta$-H\"older smooth densities up to polylog factors provided the discriminator class $\cal D$ includes the aforementioned multilayer perceptron and has VC-dimension at most polylog in $1/\epsilon$. This observation essentially generalizes \Cref{thm:intro}, which shows that if the discriminator class includes only the \emph{single} layer (with one neuron) perceptron then the best possible rate is $n^{-\beta/(2\beta+d+1)}$.

\subsection{A Stopping Criterion for Smooth Density Estimation}\label{sec:stopping crit}

As a corollary to our results, we propose a stopping criterion for training density estimators. Before doing so, let us record a result about the concentration properties of the empirical energy distance about its expectation. 

\begin{lemma}\label{lem:subgauss concentration}
    Let $\nu$ be supported on a compact subset $\Omega\subseteq\R^d$, and let $\nu_n$ be its empirical measure based on $n$ i.i.d. observations. For every $\gamma \in (0,2)$ there exists a constant $C_1 = C_1(\Omega, \gamma)>0$ such that
    \begin{equation*}
        \P\left(\cal E_\gamma(\nu,\nu_n) \geq \frac{C_1}{\sqrt n} + t\right) \leq 2\exp\left(-\frac{nt^2}{C_1}\right).
    \end{equation*}
    In other words, $\cal E_\gamma(\nu,\nu_n)$ is $O(1/n)$-subGaussian. 
\end{lemma}

\begin{proof}
    Recall the MMD formulation of the generalized energy distance from \Cref{sec:mmd form}. The corresponding kernel is given by $k_\gamma(x,y) = \|x\|^\gamma + \|y\|^\gamma - \|x-y\|^\gamma$. Clearly $$\sup_{x,x',y,y'\in\operatorname{supp}(\nu)}(k_\gamma(x,y)-k_\gamma(x',y')) \lesssim \operatorname{diam}(\Omega)^\gamma.$$Therefore, by McDiarmid's inequality we know that $\cal E_\gamma(\nu, \nu_n)$ is $O(1/n)$-subGaussian (note we don't track constants depending on $\Omega$ here). From \Cref{lem:d_a concentration} we know that $\E\cal E_\gamma(\nu,\nu_n) \lesssim 1/\sqrt{n}$, and the conclusion follows. 
\end{proof}

Consider the following scenario: we have i.i.d. training data $X_1,\dots,X_n$ from some
distribution $\nu$ and we are training an arbitrary generative model to estimate $\nu$. Suppose
that this training process gives us a sequence of density estimators $\{\mu_k\}_{k\geq1}$, which
could be the result of, say, subsequent gradient descent steps on our parametric class of
generators. 
Is there any way to figure out after how many steps $K$ we may stop the training process? In other
words, can we identify a value of $K$ such that $\TV(\nu,\mu_K)$ is guaranteed to be less than
some threshold with probability $1-\delta$? Note that an additional difficulty here is that our
generative model for $\mu_k$ is able to generate the samples from $\mu_k$ but otherwise gives us
no other access to $\mu_k$. The fast (dimension-free) concentration properties of $\cal E_\gamma$
and the minimax optimality guarantees of its minimizer (whenever $\nu$ is smooth) make it an
excellent choice for such a stopping criteria. 

Let $\nu \in \cal P_S(\beta,d, C)$ and let $\nu_n$ be its empirical version based on the $n$ i.i.d. observations. Assume further that $\{\mu_k\}_{k\geq1} \subseteq\cal P_S(\beta,d, C)$ is a sequence of density estimators based on the sample $X_1,\dots,X_n$. Finally, given the training sample $(X_1,\dots,X_n)$, for each $k$ let $\mu_{k,m_k} = \frac{1}{m_k} \sum_{i=1}^{m_k}\delta_{X^{(k)}_i}$ be the empirical distribution of the sample $(X^{(k)}_1,\dots,X^{(k)}_{m_k}) \sim \mu_k^{\otimes m_k} $. 
\begin{proposition}\label{prop:stopping crit}
    For any $\beta>0,d\geq1$ and $\gamma \in (0,2)$ there exists a constant $c = c(\beta, d, \gamma) $ such that 
    \begin{equation*}
        \P\left(\TV(\mu_k, \nu) \leq c
        \left(\sqrt{\frac{\log(1/\delta)}{n}} + \cal E_\gamma(\mu_{k,m_k}, \nu_n)
        \right)^{\frac{2\beta}{2\beta+d+\gamma}},\,\,\forall\,k\geq1\right) \geq 1- 2\delta
    \end{equation*}
    provided we take $m_k = c n \log(k^2/\delta)/\log(1/\delta)$. 
\end{proposition}

\begin{proof}
    Let $c=C_1$ where $C_1$ is as in \Cref{lem:subgauss concentration} and fix $\delta\in(0,1)$. Define the event 
    $A = \left\{\cal E_\gamma(\nu, \nu_n) \geq \frac{c}{\sqrt n} +  \sqrt{\frac{c\log(2/\delta)}{n}}\right\}$
    and similarly
    \begin{equation*}
        A_k = \left\{\cal E_\gamma(\mu_{k,m_k}, \mu_k) \geq \frac{c}{\sqrt{m_k}} + \sqrt{\frac{ct_k}{m_k}}\right\}
    \end{equation*}
    for some sequence $t_1,t_2,\dots,$ and each $k\geq1$. By \Cref{lem:subgauss concentration},
\begin{align*}
\P(A) &\leq \delta, \\
    \P(A_k) &= \E \P(A_k | X_1,\dots,X_n) \leq 2\exp(-t_k).
\end{align*}
Taking $t_k = \log(k^2\pi^2/(3\delta))$, the union bound gives
\begin{align*}
    \P\left(A \cup \bigcup_{k\geq1} A_k\right) \leq 2\delta. 
\end{align*}
By the inequality $\cal E_\gamma(\mu_k, \nu) \leq \cal E_\gamma(\mu_k, \mu_{k,m_k}) + \cal E_\gamma(\mu_{k,m_k}, \nu_n)+\cal E_\gamma(\nu_n, \nu)$ it follows that 
\begin{align*}
    &\P\left(\exists k\,:\, \cal E_\gamma(\nu, \mu_k) >  \cal E_\gamma(\mu_{k,m_k}, \nu_n) + \frac{c}{\sqrt n} + \frac{c}{\sqrt m_k} + \sqrt{\frac{c\log(2/\delta)}{n}} + \sqrt{\frac{c\log(k^2\pi^2/(3\delta))}{m_k}}  \right) \\&\qquad\leq 2\delta. 
\end{align*}
Thus, by choosing $m_k \asymp n \log(k^2/\delta)/\log(1/\delta)$ we can conclude that there exists a constant $c'$ depending only on $\beta,d,\gamma$ such that 
\begin{equation*}
    \P\left(\cal E_\gamma(\nu, \mu_k) \leq c'\sqrt{\frac{\log(1/\delta)}{n}} +  \cal E_\gamma(\mu_{k,m_k}, \nu_n),\forall k\geq 1 \right) \geq 1-2\delta. 
\end{equation*}
The final conclusion follows from \Cref{prop:d_a comparison}. 
\end{proof}

Note that our bound on the probability holds for all $k$ simultaneously, which is made possible by the fact that $m_k$ grows as $k \rightarrow \infty$. The empirical relevance of such a result is immediate: suppose we have proposed candidate generative models $\mu_1, \mu_2, \ldots$ (e.g. one after each period of training epochs, or from different training models) that is trained on an i.i.d. dataset $X_1,\dots,X_n$ of size $n$ from $\nu \in \cal P_S(\beta,d, C)$. A ``verifier'' only needs to request for $m_k$ independent draws from the $k$'th candidate, and if we ever achieve $\cal E_{(\log n)^{-1}}\left(\mu_{k, m_k}, \nu_n\right) \lesssim \sqrt{\log(1/\delta)/n}$ we can stop training and claim by \Cref{cor:d_a density est} that we are a constant factor away from (near-)minimax optimality with probability $1-\delta$.

\section{Suboptimality for Two-Sample Testing}\label{sec: negative}

So far in this paper we have shown how the empirical energy distance minimizer, while being mismatched
with the target total variation loss, nevertheless achieves nearly minimax optimal performance
for density estimation tasks. Unfortunately, this surprising effect does not carry over to other
statistical tasks, such as two-sample testing, which we describe in this section.

The task of two-sample testing over a family of distributions $\cal P$ is the following. Given two samples $(X, Y)\sim p^{\otimes n}\otimes q^{\otimes m}$ with unknown distribution, we need to distinguish between the hypotheses
$$H_0: p=q\text{ and } p \in \cal P, \quad\text{versus}\quad H_1: \TV(p, q)>\eps, \text{and }p, q \in \cal P$$
with vanishing type-I and type-II error. The special case of $m=\infty$ is known as \textit{goodness-of-fit} testing, and for the class of smooth distributions it was famously solved by \cite{Ingster87Minimax}, who showed that in dimension $d=1$ the problem is solvable with probability $1-o(1)$ if and only if 
\begin{equation}\label{eq:ingster}
   n = \omega(\epsilon^{-\frac{2\beta+d/2}{\beta}}),  
\end{equation}
in which case a variant of the $\chi^2$-test works. The case of general $m,n$ and $d\geq1$ was resolved in~\cite{arias2018remember} who showed that the problem is solvable if and only if~\eqref{eq:ingster} holds with $n$ replaced by $\min\{n,m\}$, using the very same $\chi^2$-test; see also \cite{li2019optimality}. In the remainder of the section we focus on the $m=n$ case for simplicity. 

In a recent paper \citep{paik2023maximum}, the following test statistic for two-sample testing was proposed: 
$$T_{d, k}(p,q)=\max _{(w, b) \in \mathbb{S}^{d-1} \times[0, \infty)}\left|\E_{X\sim p}\left(w^{\top} X-b\right)_{+}^k-\E_{Y\sim q}\left(w^{\top} Y-b\right)_{+}^k\right|$$
where the arguments $X, Y$ can be either discrete (e.g. via observed samples) or continuous densities. Note that here we take $(a)^0_+ = \one\{a \geq 0\}$ by convention. Specifically, the test proposed is to reject the null hypothesis when 
\begin{equation}\label{eqn: paik tst}
    T_{d,k}(p_n, q_n) \geq t_{n}, 
\end{equation}
where $p_n = \frac1n\sum_{i=1}^n\delta_{X_i}, q_n=\frac1n\sum_{i=1}^n\delta_{Y_i}$ are empirical measures and the threshold that satisfies both $t_{n} = o(1)$ and $t_n = \omega(1/\sqrt{n})$. 
One of their main technical results~\cite[Theorem 6]{paik2023maximum} asserts that the test \eqref{eqn: paik tst} returns the correct hypothesis with probability $1-o(1)$ asymptotically as $n\to\infty$ for any
qualifying sequence $\{t_n\}_{n\geq1}$ and fixed $p,q$. However, this result leaves open questions about the sample complexity of
their test, and in particular, whether it is able to achieve known minimax rates. It turns out that our results imply that their test, at least in the $k=0$ case, 
cannot attain the optimal two-sample testing sample complexity~\eqref{eq:ingster} over the smooth class $\cal P_S(\beta,d, C)$. To connect to our results, notice that 
 $$ T_{d,0}(p,q) = \overline{d_H}(p,q)\,.$$

\begin{proposition}\label{prop: tst bad}
    For all $d, \beta>0$, there exists constants $c=c(d, \beta),c' = c'(d, \beta)$ such that for all $\eps>0$, there
    exists probability density functions $p, q$ supported on the $d$-dim unit ball such that 
    \begin{enumerate}
    \item $\|p\|_{\beta, 2}, \|q\|_{\beta, 2}< c$,
    \item     $\|p-q\|_1\asymp\|p-q\|_2\asymp\eps$, and
    \item the expected test statistic satisfies  
    $$\E[T_{d, 0}(p_n, q_n)]\le 
    \frac{c'}{\sqrt{n}}\,$$for any $n\leq (\log {1\over \eps})^{-d} \eps^{-\frac{2\beta+d+1}{\beta}}$. 
    \end{enumerate}
\end{proposition}

In other words, consistent testing using the statistic $T_{d,0}$ is impossible with $n=\tilde
o(\eps^{-\frac{2\beta+d+1}{\beta}})$ samples, which is a far cry from the optimal sample complexity
\eqref{eq:ingster} attainable by the $\chi^2$ test. The proof of \Cref{prop: tst bad} is given at \Cref{sec: proof tst bad}.

\section{Conclusion}\label{sec:conclusion}

We analyzed the simple discriminating class of affine classifiers and proved its effectiveness in the ERM-GAN setting \eqref{eqn:intro adversarial} within the Sobolev class $\mathcal{P}_S(\beta,d)$ and Gaussian mixtures $\mathcal{P}_G(d)$ with respect to the $L^2$ norm (see \Cref{cor:d_a density est,prop:optimize_alpha}) and the total variation distance (see \Cref{thm:intro}). Our findings affirm the rate's near-optimality for the considered classes of $\mathcal{P}_S$ and $\mathcal{P}_G$. Moreover, we present inequalities that interlink the $\cal E_\gamma$, TV, and $L^2$ distances, and demonstrate (in some cases) the tightness of these relationships via corresponding lower bound constructions (\Cref{sec: lowers}). We also interpret the generalized energy distance in several ways that help advocate for its use in real applications. This work connects to a broader literature on the theoretical analysis of GAN-style models.

An interesting question emerges about the interaction between the expressiveness and concentration of the discriminator class. We found that the class of affine classifiers $\cal D_1$ is guaranteed to maintain some (potentially small) proportion of the total variation distance, and that it decays at the parametric rate between population and empirical distributions. 
Thus, we have traded off expressiveness for better concentration of the resulting IPM. As discussed in \Cref{sec:related works}, Yatracos' estimator lies at the other end of this discriminator expressiveness-concentration trade-off: the distance $d_Y$ is as expressive as total variation when restricted to the generator class $\cal G$, but $\sup_{\nu\in\cal G}\E d_Y(\nu,\nu_n)$ decays strictly slower than $1/\sqrt n$ for nonparametric classes $\cal G$. A downside compared to $\overline{d_H}$ is that $(i)$ the Yatracos class $\cal Y$ requires knowledge of $\cal G$ while our $\cal D_1$ is oblivious to $\cal G$ and $(ii)$ the distance $d_Y$ is impractical to compute as it requires a covering of $\cal G$. Our question is: is it possible to find a class of sets $\cal S \subseteq 2^{\bb B(0,1)}$ that lies at an intermediate point on this trade-off? In other words, does $\cal S$ exist such that the ERM $\tilde\nu$ defined in \eqref{eqn:intro adversarial} using the discriminator class $\cal D = \cal S$ is optimal over, say, $\cal G = \cal P_S$ and the induced distance converges slower than $1/\sqrt n$ but faster than $n^{-\beta/(2\beta+d)}$ between empirical and population measures? Would there be desiderata for a sample-efficient discriminator that has neither full expressiveness against total variation and does not concentrate at a parametric rate?


\bibliographystyle{alpha}
\bibliography{bib}

\appendix

\newpage
\section{Related works showing ``neural discriminator'' implies distance is small}\label{appendix:related_works}

In \Cref{sec:related works} we already reviewed at a high level existing results on GANs and
density estimation. Here we provide full details and emphasize differences with our work.


\paragraph{Results on smooth pushforward densities}

A range of works has focused on showing results for not smooth densities but smooth pushforwards $g_\#U$ (for example $U\sim\operatorname{Unif}([0, 1]^d)$ and $g$ is $\beta$-H\"older smooth).
While the exact setting differs, it can be shown that at least for TV estimation on the open unit cube (or unit ball), smooth pushforward estimations imply smooth density estimation. Indeed, one can first assume without loss of generality that the target density is bounded below via a linear transform $f\to \frac{1}{2}(\operatorname{unif}+f)$ (see e.g. Lemma 32.18 in \cite{yuryyihongbook}) and bounded above from smoothness. Furthermore, for any two (H\"older) $\beta$-smooth densities bounded above and below on (open) uniformly convex domains, the unique optimal transport map with respect to the quadratic cost is $(\beta+1)$-smooth (see Lemma 10 in \cite{chae2023likelihood} as well as Theorem 4.14 in \cite{villani2003topics}). Therefore, these rates could be compared to our results.

\begin{enumerate}
    \item (Concurrent with ours) \cite{stephanovitch2023wasserstein},
    Model 1 and 3.

     This model assumes that the true density is $g_\# U$ for U uniform on $d$-dim torus and $g$ is $(\beta+1)$-H\"older smooth from the torus to $\mathbb R^p$. The authors obtain optimal rates with respect to discriminator class consisting of $\gamma$-smooth functions for $\gamma \geq 1$ (the case of $\gamma=1$ matches with Wasserstein $W_1$). Under the assumption that the target density is $\beta$-smooth, results for comparing $W_1$ and $\TV$ (\cite{chae2024wasserstein} and Corollary 4.4 in \cite{stephanovitch2023wasserstein}) apply. By substituting in the rate of $W_1$ distance ($\gamma=1$) in Theorem 5.8 (Model 3 therein) $\widetilde O\left(n^{-\frac{\beta+1}{2 \beta+d}}\right)$ with \cite{CHAE2020108771}, one gets the optimal rate $\widetilde O\left(n^{-\frac{\beta}{2 \beta+d}}\right)$.
     However, their discriminator crucially rely on knowledge of parameters $n, \beta$ (see (4.5) therein)
    
    \item \cite{chae2022rates} 
    
    The authors obtained results for $W_1$ distance of smooth pushforwards where the observed
    samples are Gaussian-noised with variance $\sigma$. In the zero-noise case, their convergence
    results in $W_1$ on $\beta$-H\"older smooth pushforward functions give slightly suboptimal
    rates of $O\left(n^{-\frac{\beta}{2 \beta+d}}\right)$ (Theorem 3.2 therein) compared to the
    (conjectured) optimal rate $O\left(n^{-\frac{\beta}{2 \beta+d-2}}\right)$ (Theorem 4.1
    therein). As above, their results on $W_1$ could be applied to total variation in this setting
    by known distance comparison inequalities. Furthermore, their discriminator is a large
    non-parametric class (not a neural network).

    \item \cite{belomestny2021rates} 
    
    The authors studied the GAN estimator on $U\sim\operatorname{Unif}([0, 1]^d)$ and the
    pushforward map $g$ is $\Lambda$-regular and $(\beta+1)$-H\"older smooth (despite having the same minimax rate, this class do \emph{not} cover all bounded $\beta$-smooth densities, see discussion after Theorem 3 therein). Their results (Theorem
    2) obtained (log)optimal JS divergence rates, equivalent to TV since $
    \TV^2\lesssim\text{JS}$.
	In contrast to our work, their discriminator class is essentially a non-parametric
    class (of ${p(x)\over q(x) + p(x)}$, where $p,q$ range over pairs of smooth densities) that is
    approximated by giant neural networks, for which they assumed existence of an oracle
    approximating smooth discriminator functions over those networks (\cite{belomestny2023simultaneous}). Furthermore, this discriminator network needs to be larger than some function of $\beta$ and $n$ whereas ours is fixed.


\end{enumerate}

\paragraph{Results on smooth densities}
There is also a large literature on GAN estimation of smooth densities (with slightly varying
definitions) that are useful to review for contrasting with our own work.

\begin{enumerate}

    \item \cite{singh2018nonparametric} and \cite{uppal2019nonparametric}

     While these works focus on projection density estimators, Theorem 7 in \cite{singh2018nonparametric} and Theorem 9 in \cite{uppal2019nonparametric} have shown smooth density estimation rates on GANs. The authors have shown that if the discriminator class (functions of Sobolev smoothness and Besov smoothness $s$, resp.) is replaced with a neural network then the GAN estimator has optimal rate \underline{if} the neural networks well approximate the smoothness class, for which known results such as \cite{yarotsky2017error} applies. 
     However, bounds on $\TV$ or $W_1$ are not directly covered in their theorems.

     \item\cite{chen2020distribution}

    The authors considered H\"older-smooth densities on convex support that are bounded below.
    They obtained sub-optimal rates  $O\left(n^{-\frac{\gamma}{2 \gamma+d}} \vee n^{-\frac{\beta+1}{2(\beta+1)+d}}\right)$  on the IPM for H\"older smooth functions ${\mathcal{H}^\gamma}, \gamma\geq 1$ (where $\gamma=1$ transfers to Lipschitz functions with $W_1$ being the respective IPM) versus the minimax rate $O\left(n^{-\frac{\beta+\gamma}{2 \beta+d}} \vee n^{-\frac{1}{2}}\right)$. They also had to rely on (oracle) neural network universal approximating discriminators that depends on $\beta$.
\end{enumerate}

To conclude this section, we mention that there are also several recent works on non-GAN
estimators that we omitted, such as \cite[Model 2]{stephanovitch2023wasserstein}, \cite[Theorem 4]{liang2021well}, \cite{divol2022measure} and
\cite{tang2023minimax}. We recommend survey in Section 2.3 and Table 1 of
\cite{stephanovitch2023wasserstein}.

\section{Auxiliary Technical Results}
In this section we list some technical lemmas that are used in our later proofs. 

\begin{theorem}[Plancherel theorem]\label{thm:plancherel}
    Let $f,g \in L^2(\R^d)$. Then 
    \begin{equation*}
        \int_{\R^d} f(x) \overline{g(x)} \D x = \frac{1}{(2\pi)^d} \int_{\R^d} \widehat{f}(\omega) \overline{\widehat{g}(\omega)} \D \omega. 
    \end{equation*}
\end{theorem}

\begin{lemma}\label{lem:log inverse}
    Suppose $t,x,y > 0$. Then there exist finite $t$-dependent constants $C_1,C_2$ such that 
    \begin{equation*}
        x \leq \ y \log(3+1/y)^t \implies \frac{x}{\log(3+1/x)^t} \leq C_1\ y \implies x \leq C_2\ y \log(3+1/y)^t. 
    \end{equation*}
\end{lemma}

\begin{proof}
    Let us focus on the first implication. If $x \leq y$, then it clearly holds. If $y \leq x \leq y\log(3+1/y)^t$ then it suffices to show
    \begin{align*}
        \frac{x}{\log(3+1/x)^t} \leq
        \frac{y\log(3+1/y)^t}{\log(3 + 1/(y\log(3+1/y)^t))^t} \stackrel{!}{\leq} C_1 y.
    \end{align*}
    The second inequality is equivalent to 
    \begin{equation*}
    3+1/y \leq (3+1/(y\log(3+1/y)^t))^{\sqrt[t]{C_1}}. 
    \end{equation*}
    Now, if $y\geq1/2$ then clearly taking $C_1 = \log_3(5)^t$ works. Suppose that instead $y\in(0,1/2)$. Then, since $\log$ grows slower than any polynomial, there exists a $t$-dependent constant $c_t<\infty$ such that $\log(3+1/y) \leq c_t y^{-1/(2t)}$ for all $y\in(0,1/2)$. Therefore, we have
    \begin{equation*}
        3 + \frac{1}{y\log(3+1/y)^t} \geq 3 + \frac{1}{c_t^ty^{1/2}}. 
    \end{equation*}
    It is then clear that 
    \begin{equation*}
        3+\frac1y \leq \left(3 + \frac{1}{c_t^ty^{1/2}}\right)^{\sqrt[t]{C_1}}
    \end{equation*}
    holds for all $y\in(0,1/2)$ if we take $C_1$ large enough in terms of $t$. The second implication follows analogously and we omit its proof. 
\end{proof}

\begin{lemma}\label{lem:1+2a integrability}
    Let $\mu$ be a probability distribution on $\R^d$ and $\gamma \in (0,2)$. Then 
    \begin{equation*}
        \E_{X\sim\mu} \int_{\R^d} \frac{(\cos\langle \w,X\rangle-1)^2 + \sin^2\langle\w,X\rangle}{\|\w\|^{d+\gamma}} \D\w \leq \frac{16\pi^{d/2}M_{\gamma}(\mu)}{\Gamma(d/2)\gamma(2-\gamma)}. 
    \end{equation*}
\end{lemma}
\begin{proof}
    We use the inequalities $(\cos t-1)^2 + \sin^2(t) \leq 4(t^2\land 1)$ valid for all $t\in\R$. Plugging in and using the Cauchy-Schwarz inequality, the quantity on the left hand side can be bounded as
    \begin{align*}
        4\E\int_{\R^d} \frac{1\land(\|\w\|^2\|X\|^2)}{\|\w\|^{d+\gamma}}\D\w &\leq 4\operatorname{vol}_{d-1}(\bb S^{d-1})\E \int_0^\infty \frac{1\land(r^2\|X\|^2)}{r^{1+\gamma}}\D r \\
        &= \frac{8\pi^{d/2}}{\Gamma(\frac d2)}\E\left\{\|X\|^2\int_0^{\|X\|^{-1}} \frac{1}{r^{\gamma-1}} \D r + \int_{\|X\|^{-1}}^\infty \frac{1}{r^{1+\gamma}}\D r\right\} \\
        &= \frac{16\pi^{d/2}M_\gamma(\nu)}{\Gamma(\frac d2)\gamma(2-\gamma)}, 
\end{align*}
where $\vol_{d-1}(\bb S^{d-1})={\frac {2\pi ^{d/2}}{\Gamma {\bigl (}{\frac {d}{2}}{\bigr )}}}$ is the surface area of the unit $(d-1)$-sphere.
\end{proof}

\begin{lemma}\label{lem:phi_a dct bounds}
For $\gamma \in (0,2)$ define  
\begin{align*}
    B_\gamma =
    \begin{cases}
        \sup_{0<a<c}\left|\int_a^c\frac{\sin(\w)}{\w}\D{\w}\right|
        &\text{if }\gamma=1,
        \\
        \sup_{0<a<c}\left|\int_a^c\frac{\cos(\w)}{\w^{(1+\gamma)/2}}\D{\w}\right|
        &\text{if }\gamma\in(0,1),
        \\
        \sup_{0<a<c}\left|\int_a^c\frac{\cos(\w)-1}{\w^{(1+\gamma)/2}}\D{\w}\right|
        &\text{if }\gamma\in(1,2).
    \end{cases}
\end{align*}
Then $B_\gamma < \infty$. 
\end{lemma}
\begin{proof}

    In the $\gamma>1$ case, one immediately has $B_\gamma\leq \int_0^\infty\left|\frac{\min\{1, \omega^2/2\}}{\w^{(1+\gamma)/2}}\right|\D{\w}<\infty.$ Now let us consider $\gamma\leq 1$.
    One has that $B_\gamma=\sup_{0<c_1< c_2}|I_\gamma(c_2)-I_\gamma(c_1)|\leq \sup_{c>0}2|I_\gamma(c)|$ where\begin{align*}
    I_\gamma(a) =
    \begin{cases}
        \int_1^a\frac{\sin(\w)}{\w}\D{\w}
        &\text{if }\gamma=1,
        \\
        \int_1^a\frac{\cos(\w)}{\w^{(1+\gamma)/2}}\D{\w}
        &\text{if }\gamma\in(0,1),
    \end{cases}
\end{align*}
Clearly $a\mapsto I_\gamma(a)$ is continuous on $(0,\infty)$ for each $\gamma \in (0,2)$. Moreover, $$|I_{\gamma}(a)|\leq |a-1|\cdot \max\{a^{-(\gamma+1)/2}, 1\}.$$Therefore, we only need to show that $\lim_{a\to \infty}|I_\gamma(a)|$ and $\lim_{a\to 0}|I_\gamma(a)|$ are both finite.
Let $g_\gamma(x)=x^{-(\gamma+1)/2}$ and $f_\gamma(x)=\sin(x)$ if $\gamma=1$ and $\cos(x)$ otherwise, so that we may write $I_\gamma(a)=\int_a^1f_\gamma(\w)g_\gamma(\w)\D\w$.
\begin{enumerate}
    \item For ${a\to \infty}$, since $g_\gamma(\infty)=0$ and $|\int_{1}^{a}f_\gamma(x)\D t|\leq 2$ is uniformly bounded, $\int_{1}^\infty f(\w)g(\w)\D\w$ converges to a finite value by Dirichlet's test for improper integrals \cite[page 391]{malik1992mathematical}.\footnote{Reference pointed out by user Siminore on math.stackexchange.com.}
    \item For $a\to 0$, the conclusion follows by the inequality $|\sin(\omega)|\leq\min\{|\omega|,1\}$ in the case $\gamma=1$, and by  $|I_\gamma(a)|\leq \int_{a}^1\w^{-(1+\gamma)/2}\d\w\leq \int_{0}^1\w^{-(1+\gamma)/2}\d\w=\frac{2}{1-\gamma}$ for $\gamma < 1$. 
\end{enumerate}
Therefore, $\sup_{a>0}|I_\gamma(a)|<\infty$ which concludes the proof.

\end{proof}

\begin{lemma}\label{lem:psi_a fourier transform stronger}
Let $ \int_0^\infty \cdot\,\D \w \eqdef \lim_{\epsilon\to0}\int_{1/\epsilon\geq \w \geq\epsilon}\cdot\,\D \w$ and recall the definition of $\psi_\gamma$ from \eqref{eqn:psi_gamma def}. Then, for $x\neq0$ the following hold:
    \begin{align*}
            \psi_\gamma(x) = C_{\psi_\gamma} \begin{cases} \int_0^\infty\frac{\sin(\w x)}{\omega}\D\omega + \frac\pi2&\text{for }\gamma=1, \\
            \int_0^\infty \frac{\cos(\w x)}{\omega^{(1+\gamma)/2}} \D\omega 
            &\text{for } \gamma\in(0,1), \\
            \int_0^\infty \frac{\cos(\w x)-1}{\omega^{(1+\gamma)/2}} \D\omega 
            &\text{for }\gamma\in(1,2), \end{cases}
    \end{align*}
    where 
    \begin{equation*}
        C_{\psi_\gamma} = \begin{cases}
        \left(\cos(\frac{\pi(\gamma-1)}{4})\Gamma(\frac{1-\gamma}{2})\right)^{-1}
        &\text{if }\gamma\neq1,
        \\ 
        \frac1\pi
        &\text{if }\gamma=1.
        \end{cases}
    \end{equation*}
\end{lemma}
\begin{proof}
For $x\neq0$ clearly
    \begin{align*}
        \int_0^\infty \frac{\sin(\w x)}{w}\D{w}=\sgn(x)\, \int_0^\infty \frac{\sin(\w)}{w}\D{w} = \sgn(x)\frac{\pi}{2}, 
    \end{align*}
    which shows the first claim. Assume from here on without loss of generality that $x>0$. For $\gamma\in(0,1)$, by the residue theorem, 
    \begin{align*}
        \int_0^\infty \frac{\cos(\w x)}{\w^{(1+\gamma)/2}}\D{\w}
        &=x^{(\gamma-1)/2}\, \int_0^\infty \Re\left(
        \frac{e^{i\w}}{\w^{(1+\gamma)/2}}
        \right)
        \D{w} 
        \\
        &=x^{(\gamma-1)/2}\, \Re\left(ie^{-i\frac{\pi}{2}\gamma}\right)  \int_0^\infty \frac{e^{-z}}{z^{(1+\gamma)/2}}\D{z} 
        \\
        &=x^{(\gamma-1)/2}\, \cos\left(\frac{\pi(\gamma-1)}{4}\right) \Gamma((1-\gamma)/2).
    \end{align*}
    Similarly, for $\gamma\in(1,2)$, integration by parts and the residue theorem gives
    \begin{align*}
         \int_0^\infty \frac{\cos(\w x)-1}{\w^{(1+\gamma)/2}}\D{\w}
        &=x^{(\gamma-1)/2}\, \int_0^\infty
        (\cos(\w)-1)
        \D \left(\frac{-1}{((\gamma-1)/2) \,\w^{(\gamma-1)/2}}\right)
        \\
        &=-x^{(\gamma-1)/2} \int_0^\infty
        \frac{\sin(w)}{(\gamma-1)/2\,\w^{(\gamma-1)/2}}
        \D{\w}
        \\
        &=-x^{(\gamma-1)/2} \frac{2}{\gamma-1} \int_0^\infty\Im\left(
        \frac{e^{iw}}{\w^{(\gamma-1)/2}}\right)
        \D{\w}
        \\
        &=-x^{(\gamma-1)/2} \frac{2}{\gamma-1}
        \Im\left(ie^{-\frac{\pi}{2}(\gamma-1)/2}\right)\int_0^\infty
        \frac{e^{-z}}{z^{(\gamma-1)/2}}
        \D{z}
        \\
        &=-x^{(\gamma-1)/2} \frac{2}{\gamma-1}\cos\left(\frac{\pi(\gamma-1)}{4}\right)\Gamma(1-(\gamma-1)/2) \\
        &= x^{(\gamma-1)/2}\,\cos\left(\frac{\pi(\gamma-1)}{4}\right)\Gamma((1-\gamma)/2). 
    \end{align*}
\end{proof}

\begin{lemma}
\label{lem: Dbeta of GMM}
    Let $\phi$ be the probability density function of $\mc{N}(0,\sigma I_d)$ and write $\widehat{\phi}$ for its Fourier transform. Then, for any $\beta\geq 1$,
    \begin{align*}
        \| \widehat{\phi}(\w) \|\w\|^\beta \|_2^2=\frac{\pi^{d/2}}{\Gamma(d/2)\sigma^{2\beta+d}}\Gamma\left(\frac{2\beta+d}{2}\right) \leq \frac{5\pi^{d/2}}{\Gamma(d/2)\sigma^{2\beta+d}} \left(\frac{2\beta+d}{2e}\right)^{\frac{2\beta+d-1}{2}}.
    \end{align*}
\end{lemma}
\begin{proof}
    It is well known that $\widehat{\phi}(\w)= e^{-\frac{\sigma^2}{2}\|\w\|^2}$. The claimed equality then follows from the formula for the $2\beta$'th moment of the Gaussian distribution with mean $0$ and variance $1/(2\sigma^2)$. The inequality follows by \Cref{lem:gamma fn}. 
 
\end{proof}

\begin{lemma}[Properties of the gamma function]\label{lem:gamma fn}
    For all $x>1$ the inequality $\Gamma(x) \leq 5(x/e)^{x-1/2}$ holds. 
\end{lemma}
\begin{proof}
    In \cite{minc1964some} authors showed that $\log\Gamma(x) \leq (x-\frac12)\log(x)-x+\frac12\log(2\pi)+1$ for all $x \geq 1$, from which the second claim follows as $\exp(\frac12\log(2\pi)+1/2)<5$.
\end{proof}

\begin{lemma}\label{lem:integral of sin}
    Let $b \geq 1$ and $|a| < b$. Then
    \begin{equation*}
        \int_0^r x^a |\sin(x)|^b\d x = O(r^{a+b})
    \end{equation*}
    as $r \to \infty$, where we hide constants depending on $a,b$. 
\end{lemma}
\begin{proof}
    Since we are only interested in the asymptotic behaviour as $r \to \infty$, assume without loss of generality that $r\geq1$. Then, we have
    \begin{equation*}
        \int_0^r x^a|\sin(x)|^b\d x = \underbrace{\int_0^1x^a|\sin(x)|^b\d x}_{I} + \underbrace{\int_1^r x^a|\sin(x)|^b\d x}_{II}. 
    \end{equation*}
    Using the inequality $|\sin(x)|\leq x$, we can bound the first term as
    \begin{equation*}
        I \leq \int_0^1 x^{a+b}\d x = \frac{1}{a+b+1} \leq 1 = O(r^{a+b}),  
    \end{equation*}
    since $a+b>0$. For the second term, we obtain
    \begin{equation*}
        II \leq \int_1^r x^a \d x = \begin{cases}\begin{rcases}
            \frac{r^{a+1}-1}{a+1} &\text{if }a\neq-1 \\
            \log(r) &\text{if }a=-1
        \end{rcases}\end{cases} = O(r^{a+b}), 
    \end{equation*}
    where the last step uses $a+b > 0$ and $b\geq1$.

\end{proof}

\begin{lemma}\label{lem: tail_int}
    Let $a,b,c\in\R$ with $b>0$ be constants. For all large enough $r$ one has
    $$\int_{r}^{\infty}x^a\exp\left(-\frac{bx}{\log^{2}(x+2)}\right)\d x<r^{-c}.$$
\end{lemma}
\begin{proof}
    Assume, without loss of generality, that $c\geq0$. For all large enough $x$ one has $$\exp\left(-\frac{bx}{\log^{2}(x+2)}\right)<x^{-a-c-2}.$$Therefore, for large enough $r$, 
    $$\int_{r}^{\infty}x^a\exp\left(-\frac{bx}{\log^{2}(x+2)}\right)\d x<\int_{r}^{\infty}x^{-c-2}\d x\asymp r^{-c-1}<r^{-c}.$$
\end{proof}

\begin{lemma}\label{lem: Bessel}
Let $J_{\nu}$ be the Bessel function of the first kind of order $\nu$.
\begin{enumerate}
    \item\label{item:Bessel sphere} For all $x\in\R^d$, $$\int_{\R^d} e^{i\langle x,v\rangle} \D\sigma(v) = (2\pi)^{d/2}\|x\|^{1-d/2}J_{d/2-1}(\|x\|).$$ 
    \item\label{item:Bessel asymp} For any $\nu \in \R$, as $x\to\infty$
     \begin{equation}\label{eqn:bessel at infty}
     J_{\nu}(x)
                =\sqrt{\frac{2}{\pi x}}
                \cos\left(x-{\frac {\nu\pi}{2}}-\frac{\pi }{4}\right)+O(x^{-3/2}).
                \end{equation}
    \item\label{item:Bessel ball} For all $x\in\R^d$, $$\int_{\bb{B}^d(0,1)} e^{i\langle x,w\rangle} \d w = \left(\frac{2\pi}{\|x\|}\right)^{d/2} J_{d/2}(\|x\|).$$
\end{enumerate}
\end{lemma}
\begin{proof}
    For \Cref{item:Bessel sphere} set $r=\|x\|$ and $s=(2\pi)^{-1}$ in the calculation on page 154 of \cite{stein1971introduction}.

    For \Cref{item:Bessel asymp} see \cite[Eq. (1) in Section 7.21]{watson1995treatise}. 
    
    For \Cref{item:Bessel ball}, we can compute
    \begin{align}\label{eqn:fourier transform of ball}
        \int_{\|w\|\leq1} e^{i\langle x,w\rangle} \d w &= \int_{-1}^1 e^{i\|x\|w_1} \int_{w_2^2 + \dots + w_d^2 \leq 1-w_1^2} \d w_2 \dots \d w_d \d w_1 \nonumber\\
        &= \frac{\pi^{(d-1)/2}}{\Gamma(\frac{d+1}{2})} \int_{-1}^1 e^{i\|x\|w_1} (1-w_1^2)^{(d-1)/2} \d w_1. 
    \end{align}
    
    Recall from \cite[Section 3.1]{watson1995treatise} the definition of the Bessel function of the first kind as
    \begin{equation*}
        J_\nu(x) = \frac{(x/2)^\nu}{\Gamma(\nu+\frac12)\Gamma(\frac12)} \int_0^\infty \cos(x\cos(\theta))\sin^{2\nu}(\theta)\d\theta
    \end{equation*}
    valid for $\nu > -1/2$; the above is also known as the Poisson representation. Changing variables to $u=\cos(\theta)$, we see that it is equal to 
    \begin{align}\label{eqn:poisson representation}
        J_\nu(x) &= \frac{(x/2)^\nu}{\Gamma(\nu+\frac12)\Gamma(\frac12)} \int_{-1}^1 e^{ixu}(1-u^2)^{\nu-1/2}\d u. 
    \end{align}
    Comparing \eqref{eqn:poisson representation} with \eqref{eqn:fourier transform of ball} concludes the proof. 
\end{proof}

\begin{lemma}\label{lem:first_h}
    There exists a radial function $h_0 \in L^2(\R^d)$ such that
    \begin{align*}
        &\supp(h_0) \subseteq \mbb{B}(0,1), &&   \\
        &|\widehat{h}_0(w)| \leq C \exp \left( -\frac{c \|w\|}{\log(\|w\|+2)^{2}} \right) && \quad \text{for all } w \in \R^d, \\
        &|\widehat{h}_0(w)| \geq \frac{1}{2} && \quad \text{for all }\|w\| \leq r_{\min},
    \end{align*}
    where $C, c, r_{\min} > 0$.
\end{lemma}

\begin{proof}
   Apply Theorem 1.4 in \cite{cohen2023fractal} using the spherically symmetric weight function $u : \R^d \to \R_{\leq 0}$ defined by
    \begin{equation*}
        u(w) = u(\|w\|) = - \frac{\|w\|}{\log(\|w\|+2)^2} \left(\frac{(\|w\|-2)_+}{\|w\|+2}\right)^4,
    \end{equation*}
    where $(a)_+ := \max(a, 0)$ for $a \in \R$. 
\end{proof}

\section{Proof of Theorem \ref{prop:slice equiv form}}

For $v\in\bb S^{d-1}$ and $b\in\R$ let $\theta_v(x) = \langle v,x\rangle$ and write $\eta_v\eqdef \theta_v\#(\mu-\nu)$ for the pushforward of the measure $\mu-\nu$ through the map $\theta_v$. To start with, we notice that
\begin{equation}\label{eq: d_a in theta}
\begin{aligned}
   &\int_{\bb S^{d-1}}\int_\R \Big[\E\psi_\gamma(\langle X,v\rangle-b)-\E\psi_\gamma(\langle Y,v\rangle-b)\Big]^2\D b\D\sigma(v)
   \\
    =&\int_{\bb S^{d-1}} \int_\R 
    \left(
        \int_\R
        \psi_\gamma(x-b)
        \D \eta_v(x)
    \right)^2
    \D b\D\sigma(v), 
\end{aligned}
\end{equation}
For each $v\in\bb S^{d-1}$, the measure $\eta_v$ has at most countably many atoms, therefore $b\mapsto \eta_v(\{b\}) = 0$ $\operatorname{Leb}$-almost everywhere. Then, by Tonelli's theorem we can conclude that $\eta_v(\{b\})=0$ for $\sigma\otimes\operatorname{Leb}$-almost every $(v,b)$, thus going forward we can focus on the case $x\neq b$. By \Cref{lem:psi_a fourier transform stronger}, and writing $A_\epsilon = [\epsilon, 1/\epsilon]$ for $\epsilon > 0$, we have
\begin{align*}
    \int_\R
    \psi_\gamma(x-b)
    \D\eta_v(x) =
    C_{\psi_\gamma} \int_\R \lim_{\eps\to0} \int_{A_\eps} \begin{rcases}\begin{dcases}
        \frac{\sin(\w(x-b))}{\omega}
        &\text{if }\gamma=1\\
        \frac{\cos(\w(x-b))}{\omega^{(1+\gamma)/2}}
        &\text{if }\gamma\in(0,1)\\
        \frac{\cos(\w(x-b))-1}{\omega^{(1+\gamma)/2}} 
        &\text{if }\gamma\in(1,2)
    \end{dcases}\end{rcases} \D\omega
        \D \eta_v(x) . 
\end{align*}
Note that in the $\gamma=1$ case we implicitly used that $\int \d\eta_v(x)=0$. To exchange the integral over $x$ and the limit over $\eps$, notice that for any $\eps>0$ and $x\neq b\in\R$,
\begin{align*}
        &\left|
        \int_{\eps}^{1/\eps}
        \frac{\sin(\w(x-b))}{\omega} 
        \D\omega
        \right|
        \leq
        B_\gamma
        &&\text{if }\gamma=1,
        \\
        &\left|
        \int_{\eps}^{1/\eps}
        \frac{\cos(\w(x-b))}{\omega^{(1+\gamma)/2}} \D\omega
        \right|
        \leq
        B_\gamma|x-b|^{(\gamma-1)/2}
        &&\text{if }\gamma\in(0,1),
        \\
        &\left|
        \int_{\eps}^{1/\eps}
        \frac{\cos(\w(x-b))-1}{\omega^{(1+\gamma)/2}} \D\omega
        \right|
        \leq
        B_\gamma|x-b|^{(\gamma-1)/2}
        &&\text{if }\gamma\in(1,2).
\end{align*}
where $B_\gamma < \infty$ depends only on $\gamma$ and is defined in \Cref{lem:phi_a dct bounds}. We now show that $\int_\R |x-b|^{(\gamma-1)/2} \D |\eta_v|(x)<\infty$ for $\sigma\otimes\operatorname{Leb}$-almost every $b,v$. To this end, let $S = \{(b,v)\in\R\times\bb S^{d-1}:\int_\R|x-b|^{(\gamma-1)/2} \D|\eta_v|(x) = \infty\}$ and assume for contradiction $(\sigma\otimes\operatorname{Leb})(S) > 0$. Then $\one_{([-B,B]\times\bb S^{d-1})\cap S} \uparrow \one_S$ as $B\to\infty$, and thus by the monotone convergence theorem there exists a finite $B$ such that $\operatorname{Leb}(([-B,B]\times\bb S^{d-1})\cap S) > 0$. However, by Tonelli's theorem we have
\begin{align*}
    \int_{-B}^B \left(\int_\R |x-b|^{(\gamma-1)/2}\D|\eta_v|(x)\right)^2\D b &\leq \int_{-B}^B \int_\R |x-b|^{\gamma-1} \D|\eta_v|(x)\D b \\
    &\leq 2\int_\R \int_0^{B+|x|} b^{\gamma-1} \D b \D|\eta_v|(x) \\
    &\lesssim \int_\R (B+|x|)^{\gamma} \D |\eta_v|(x) \\
    &\lesssim B^{\gamma} + \E_{X\sim\mu}\left[|\langle v,X\rangle|^{\gamma}\right] + \E_{Y\sim\nu}\left[|\langle v,Y\rangle|^{\gamma}\right] \\
    &\leq B^{\gamma} + M_\gamma(\mu+\nu), 
\end{align*}
which, after integration over $v\in \bb S^{d-1}$, leads to a contradiction if $M_{\gamma}(\mu+\nu)<\infty$. Continuing under the assumption $M_{\gamma}(\mu+\nu)<\infty$, we can apply the dominated convergence theorem to obtain
\begin{align*}
    \int_\R
    \psi_\gamma(x-b)
    \D\eta_v(x) =
    C_{\psi_\gamma} \lim_{\eps\to0} \int_\R  \int_{A_\eps} \begin{rcases}\begin{dcases}
        \frac{\sin(\w(x-b))}{\omega}
        &\text{if }\gamma=1\\
        \frac{\cos(\w(x-b))}{\omega^{(1+\gamma)/2}}
        &\text{if }\gamma\in(0,1)\\
        \frac{\cos(\w(x-b))-1}{\omega^{(1+\gamma)/2}} 
        &\text{if }\gamma\in(1,2)
    \end{dcases}\end{rcases} \D\omega
        \D \eta_v(x) . 
\end{align*}
Then by Fubini's theorem, we exchange the order of integration to get
\begin{align*}
    \int_\R
    \psi_\gamma(x-b)
    \D\eta_v(x) =
    C_{\psi_\gamma} \lim_{\eps\to0}  \int_{A_\eps}\int_\R  \begin{rcases}\begin{dcases}
        \frac{\sin(\w(x-b))}{\omega}
        &\text{if }\gamma=1\\
        \frac{\cos(\w(x-b))}{\omega^{(1+\gamma)/2}}
        &\text{if }\gamma\in(0,1)\\
        \frac{\cos(\w(x-b))-1}{\omega^{(1+\gamma)/2}} 
        &\text{if }\gamma\in(1,2)
    \end{dcases}\end{rcases} 
        \D \eta_v(x) \D\omega. 
\end{align*}
Notice that $\int_\R e^{-i\w x} \D\eta_v(x)=\widehat{\eta}_v(\w)$, $\widehat{\eta}_v(\w)=\overline{\widehat{\eta}_v(-\w)}$ and $\widehat\eta_v(0)=0$, 
\begin{equation*}
\begin{aligned}
    \int_\R
    \psi_\gamma(x-b)
    \D \eta_v(x) &= C_{\psi_\gamma} \lim_{\eps\to0}  \int_{A_\eps} \frac{1}{\omega^{(1+\gamma)/2}}  
    \begin{rcases}\begin{dcases}
            \Im(e^{-i\w b}\overline{\widehat{\eta}_v(\w)}) 
        &\text{if }\gamma=1 \\
        \Re(e^{-i\w b}\overline{\widehat{\eta}_v(\w)}) 
        &\text{if }\gamma\neq 1
    \end{dcases}\end{rcases} \D\omega
    \\
    &= C_{\psi_\gamma} \lim_{\eps\to0} \begin{cases}
        \Im\left(\widehat{\Psi}_{\gamma,v,\eps}(b)\right) &\text{if } \gamma=1 \\ \Re\left(\widehat{\Psi}_{\gamma,v,\eps}(b)\right) &\text{if } \gamma\neq1.
    \end{cases}
\end{aligned}
\end{equation*}
where we write $$\Psi_{\gamma,v,\eps}(\omega) = \frac{  \overline{\widehat{\eta}_v(\omega)}} {\omega^{(1+\gamma)/2}} \one\{\omega\in A_\eps\}.$$ 
Notice that $\Psi_{\gamma,v,\eps}$ is bounded and compactly supported and thus lies in $L^p(\R)$ for any $p$, and so in particular
\begin{align*}
    \Psi_{\gamma,v,\eps}
    \in L^1(\R) \cap L^2(\R),
\end{align*}
which ensures that $$\widehat{\Psi}_{\gamma,v,\eps}\in L^\infty(\R) \cap L^2(\R).$$ 
Finally, let us write $$\Psi_{\gamma,v}(\omega) = \lim_{\epsilon\to0}\Psi_{\gamma,v,\eps}(\omega) = \frac{\overline{\widehat{\eta}_v(\omega)}}{\omega^{(1+\gamma)/2}}\one\{\omega>0\}$$ for every $\w$. We now show that $\Psi_{\gamma,v} \in L^2(\R)$ provided $M_{\gamma}(\mu+\nu)<\infty$, which is assumed throughout. Let $(X,Y)\sim\mu\otimes\nu$. We have
\begin{align*}
    \int_\R |\Psi_{\gamma,v}(\w)|^2 \D\w &= \int_0^\infty \frac{|\widehat\eta_v(\w)|^2}{w^{1+\gamma}} \D w \\
    &= \int_0^\infty \frac{(\E[\cos\langle\w,X\rangle-\cos\langle\w,Y\rangle])^2+(\E[\sin\langle \w,X\rangle-\sin\langle\w,Y\rangle])^2}{\w^{1+\gamma}}\D\w.
\end{align*}
Using the inequality $(a-b)^2\leq2(a-1)^2+(b-1)^2,\,\forall\,a,b\in\R$ for the $\cos$ term, the inequality $(a+b)\leq2a^2+2b^2,\,\forall\,a,b\in\R$ for the $\sin$ term, and applying Jensen's inequality to take the expectation outside, we can conclude that $\Psi_{\gamma,v}\in L^2(\R)$ by \Cref{lem:1+2a integrability}. Thus, by the dominated convergence theorem
\begin{align*}
    \left\|
    \Psi_{\gamma,v,\epsilon}-\Psi_{\gamma,v}\right\|_2 \to 0
\end{align*}
as $\epsilon \to 0$. Then, by Parseval's identity
\begin{align}
\label{eq: L2 phi}
    \left\|
    \widehat{\Psi}_{\gamma,v,\epsilon}-\widehat{\Psi}_{\gamma,v}\right\|_2 \to 0
\end{align}
as $\epsilon\to0$. It is well know that convergence in $L^2(\R)$ implies that there exists a subsequence $\{\eps_n\}_{n=1}^\infty$ with $\eps_n\to 0$ and $\widehat{\Psi}_{\gamma,v,\epsilon}\to \widehat{\Psi}_{\gamma,v}$ almost everywhere.\footnote{We could also conclude this by Carleson’s theorem.} Therefore, by passing to this subsequence, it follows that 
\begin{align*}
    \int_\R
    \psi_\gamma(x-b)
    \D \eta_v(x)
    =
    C_{\psi_\gamma}\begin{cases} \Im\left(\widehat{\Psi}_{\gamma,v}(b)\right) &\text{if }\gamma=1 \\ \Re\left(\widehat{\Psi}_{\gamma,v}(b)\right) &\text{if }\gamma\neq1\end{cases}
\end{align*}
for $\sigma\otimes\operatorname{Leb}$-almost every $(b,v)\in\R\times\bb S^{d-1}$. 
Note that since $\eta_v(\w)\in\R$,
\begin{align}
    \Re\left(\widehat{\Psi}_{\gamma,v}(b)\right) 
    &= \frac{\widehat{\Psi}_{\gamma,v}(b)+\overline{\widehat{\Psi}_{\gamma,v}(b)}}{2} \nonumber
    \\
    &= \frac12 \int_0^\infty  \left(\frac{\widehat{\eta}_v(\omega)}{\omega^{(1+\gamma)/2}}e^{ib\w} +\frac{\widehat{\eta}_v(-\omega)}{\omega^{(1+\gamma)/2}}e^{-ib\w} \right) \d\w\nonumber
    \\
    &= \frac12 \int_{-\infty}^\infty  \frac{\widehat{\eta}_v(\omega)\operatorname{sign}(\omega)}{|\omega|^{(1+\gamma)/2}}e^{ib\w}  \d\w\nonumber
    \\
    &= \mc{F}\left[ \frac{\widehat{\eta}_v(\omega)\operatorname{sign}(\omega)}{2|\omega|^{(1+\gamma)/2}} \right](-b),\label{eqn:Re Psi_g,v final}
\end{align}
\begin{align}
    \Im\left(\widehat{\Psi}_{\gamma,v}(b)\right) 
    &= \frac{\widehat{\Psi}_{\gamma,v}(b)-\overline{\widehat{\Psi}_{\gamma,v}(b)}}{2i} \nonumber
    \\
    &= \frac1{2i} \int_0^\infty  \left(\frac{\overline{\widehat{\eta}_v(\omega)}}{\omega^{(1+\gamma)/2}}e^{-ib\w} -\frac{\overline{\widehat{\eta}_v(-\omega)}}{\omega^{(1+\gamma)/2}}e^{ib\w} \right) \d\w\nonumber
    \\
    &= \frac1{2i} \int_{-\infty}^\infty  \frac{\overline{\widehat{\eta}_v(\omega)}}{|\omega|^{(1+\gamma)/2}} \sign(\w)e^{ib\w}  \d\w\nonumber
    \\
    &= \mc{F}\left[ \frac{\overline{\widehat{\eta}_v(\omega)}\sign(\w)}{2i|\omega|^{(1+\gamma)/2}} \right](-b).\label{eqn:Im Psi_g,v final}
\end{align}

Plugging \eqref{eqn:Re Psi_g,v final} and \eqref{eqn:Im Psi_g,v final} into \eqref{eq: d_a in theta}, by Parseval's identity (implicitly using that $\Psi_{\gamma,v} \in L^2(\R)$), we obtain
\begin{align*}
    \int_{\bb S^{d-1}}\int_\R \Big[\E\psi_\gamma(\langle X,v\rangle-b)-\E\psi_\gamma(\langle Y,v\rangle-b)\Big]^2\D b\D\sigma(v)
    & =\int_{\bb S^{d-1}} \int_\R 
    \left(
        \int_\R
        \psi_\gamma(x-b)
        \D \eta_v(x)
    \right)^2
    \D b\D\sigma(v), \\
    &= 2\pi C_{\psi_\gamma}^2\int_{\bb S^{d-1}}\int_\R 
    \frac{|\widehat{\eta}_v(\omega)|^2}{4|\omega|^{1+\gamma}}
    \D w\D\sigma(v) \\
    &= \pi C_{\psi_\gamma}^2\int_{\bb S^{d-1}}\int_0^{\infty}
    \frac{|\widehat{\eta}_v(\omega)|^2}{|\omega|^{1+\gamma}}
    \D w\D\sigma(v) \\
    &= \pi C_{\psi_\gamma}^2 \int_{\R^d} \frac{|\cal F[\mu-\cal \nu](\omega)|^2}{\|\w\|^{d+\gamma}} \D\omega, 
\end{align*}
where the last step uses a polar change of variable. The result follows after comparing with \Cref{prop:norm equiv form}.

\section{Proof of Proposition \ref{prop:riesz}}
    Let $\cal{S}(\R^d)$ be the Schwartz space and $\cal{S}'(\R^d)$ be the space of all tempered distributions on $\R^d$.
    Let $\tau=\mu-\nu$ and $s=(d+\gamma)/2$. First, note that $$\int \frac{K_s(x)\d{x}}{(1+\|x\|^2)^d}<\infty,$$so by \cite[Theorem 0.10]{landkof1972foundations} we have $K_s\in \cal{S}'(\R^d)$.
    By \cite[Theorem 0.12]{landkof1972foundations}, since $K_s\in \mc{S}'(\R^d)$ and $\tau$ has compact support, 
    $$ \widehat{I_s f}=\widehat{K_{s} * \tau} =  \widehat{K_{s}} \widehat{\tau}.$$
    By Plancherel's identity, 
    $$(2\pi)^{\frac{d}{2}}\|I_s \tau\|_2 = \|\widehat{I_s \tau}\|_2=\|\widehat{K_{s}}  \widehat{\tau}\|_2 = \frac{1}{\sqrt{F_\gamma(d)}} \cal{E}_\gamma(\mu,\nu),$$
    where the last equality follows from \Cref{prop:norm equiv form}.

\section{Proof of Theorem \ref{thm:tightness} and Proposition \ref{prop:d_H tightness}}\label{sec: lowers}
In this section we prove both \Cref{thm:tightness} and \Cref{prop:d_H tightness}. To do so, we give two constructions. The first one, presented in \Cref{sec:1d lower construction}, only applies in one dimension and gives optimal results. The second construction is given in \Cref{sec: lower log construction} applies in all dimensions, but loses a polylogarithmic factor. 

\textbf{Notation:} Abusing notation, in what follows we write $\cal E_\gamma(f,g)$ and $\overline{d_H}(f,g)$ even when $f$ and $g$ are not necessarily probability measures or probability densities. We will also write $\|f\|_{t,2} = \|\|\cdot\|^t\widehat f\|_2$ for potentially negative exponents $t\in\R$. Note that $\cal{E}_\gamma(f,0)=\sqrt{F_\gamma(d)}\|\widehat f\|_{-\frac{d+\gamma}{2},2}$.

\subsection{The Case $d=1$}\label{sec:1d lower construction}
The Lemma below constructs the \emph{difference} of two densities that has favorable properties. 
\begin{lemma}\label{lem: d=1 constr}
Let $f(x) = 1\{|x|\le\pi\} \sin(rx)$ with $r\in\Z$ and write $f_\beta = f*\cdots*f$ for $f$ convolved with itself $\beta-1$ times, i.e. $f_1 = f, f_2=f*f$ and so on. Fix an integer $\beta\geq1$ and let $|t| < \beta$. We have
\begin{align}\label{eq:hb1}
    \|f_\beta\|_{t,2} \asymp r^t, 
    \|f_\beta\|_1 \asymp 1,
    \text{ and }
    \overline{d_H}(f_\beta,0) \asymp \frac{1}{r}, 
\end{align}
as $r\to\infty$ where the constants may depend on $\beta,t$. 
\end{lemma}
\begin{proof} 
The intuition for the estimates~\eqref{eq:hb1} is simple: most of the energy of $f$ (and hence
$f_\beta$) is at frequencies around $|\omega|\approx r$ and thus differentiating $t$ times
boosts the $L_2$-energy by $r^t$. A simple computation shows $\widehat f(\omega) = c {(-1)^r\over i} {r\over \omega^2 - r^2} \sin(\omega
\pi)$. Note that because $r \in  \Z$ we have $\|\widehat f\| \asymp \|\widehat f_\beta\| \asymp
1$.

\textbf{Estimating $\|f_\beta\|_{t,2}$.} By definition we have
$$ \|f_{\beta}\|^2_{t,2} \asymp \int_{0}^\infty |\widehat f(\omega)|^{2\beta} \omega^{2t} \,\d\omega\asymp \int_0^\infty {r^{2\beta}\over (\omega^2 -
r^2)^{2\beta}} \omega^{2t} \sin^{2\beta}(\omega \pi)\d\omega.$$
We decompose the integral into three regimes:
\begin{enumerate}
\item \underline{$\omega < r/2$}: here $(\omega^2 - r^2) \asymp r^2$ and thus 
\begin{align*}
\int_0^{r/2}(\dots) 
 \asymp r^{-2\beta} \int_0^{r/2} \omega^{2t} \sin^{2\beta} (\omega \pi)
        \lesssim  r^{2t}
\end{align*}
by \Cref{lem:integral of sin}. 
\item \underline{$\omega > 3r/2$}: here $(\omega^2 - r^2) \asymp \omega^2$ and thus
	$$ 
	\int_{3r/2}^{\infty}(\dots) \asymp r^{2\beta} \int_{3r/2}^{\infty} {\sin^{2\beta} (\omega
	\pi)\omega^{2t}\over
	\omega^{4\beta}} \asymp r^{2\beta} r^{2t-4\beta+1} = r^{1+2t-2\beta}\ll r^{2t}.
	$$
\item \underline{$\omega \in[ r/2, 3r/2]$}: here $(\omega^2 - r^2) \asymp yr$, where $y=\omega - r$. Note
also $\sin(\omega \pi) = \sin(r\pi + y\pi) = (-1)^r\sin(y\pi)$, and $\omega \asymp r$. Thus
	$$ 
	\int_{r/2}^{3r/2} (\dots)\D\omega = \int_{-r/2}^{r/2} (\dots)\d y  \asymp r^{2\beta} \int_{-r/2}^{r/2}
	{\sin^{2\beta} (y \pi) r^{2t}\over
	(yr)^{2\beta}} \D y\asymp r^{2t} \int_\R \left(\frac{\sin(y\pi)}{y}\right)^{2\beta} \D y \asymp r^{2t}. 
	$$
	where the last inequality follows by that the integrand is bounded at $0$ and
	has $y^{-2\beta}\lesssim y^{-2}$ tail. 
\end{enumerate}

\textbf{Estimating $\|f_\beta\|_1$.} Follows from $\|f_\beta\|_1\lesssim \|f_\beta\|_2 \asymp 1$ by the Cauchy-Schwartz inequality and $\|f_\beta\|_1\geq \|\widehat{f_\beta}\|_\infty\asymp 1$ by the Hausdorff–Young inequality.

\textbf{Estimating $\overline{d_H}$.} We get $\overline{d_H}(f_\beta,0)\gtrsim \cal{E}_1(f_\beta,0)\asymp\|f_\beta\|_{-1,2}\asymp\frac1r$ from the first estimate. For the upper bound, note that $\widehat{\sign(x)}=\frac{2}{i\w}$ and $\overline{d_H}(f_\beta,0)=\sup_{b}\frac12 \int f_\beta(x)\sign(x-b)\d{x}$, so by Plancherel's identity, 
\begin{align*}
    \overline{d_H}(f_\beta,0)
    \lesssim \sup_{b} \int \left|\widehat{f_\beta}(\w)\frac{e^{ib\w}}{\w}\right| \d{\w}
    \lesssim \int_0^\infty {r^{\beta}\over (\omega^2 - r^2)^{\beta}} \omega^{-1} \sin^{\beta}(\omega \pi). 
\end{align*}
The fact that the above is $O(1/r)$ follows analogously to the proof of our bound on $\|f_\beta\|_{t,2}$ so we omit it. This concludes our proof. 

\end{proof}

\begin{proof}[Proof of \Cref{thm:tightness} and \Cref{prop:d_H tightness} for $d=1$] We now turn to showing   tightness in one dimension, utilizing the density difference constructed in \Cref{lem: d=1 constr}. Given a value of the smoothness $\beta > 0$, set $\overline\beta=\lceil\beta\rceil+1$ and let $f_{\overline\beta}$ be as in \Cref{lem: d=1 constr} with $r=\epsilon^{-1/\beta}$ for some $\epsilon \in (0,1)$. Let $p_0$ be a smooth, compactly supported density with $\inf_{x\in[-\pi,\pi]}p_0(x) > 0$. Define
\begin{equation*}
    p_\epsilon(x) = p_0(x) + \epsilon f_{\overline\beta}(\overline\beta x)/2\qquad\text{and}\qquad q_\epsilon(x) = p_0(x) -\epsilon f_{\overline\beta}(\overline\beta x)/2.
\end{equation*}
Clearly both $p_\epsilon, q_\epsilon$ are compactly supported probability densities for sufficiently small $\epsilon$, since $\|f_{\overline\beta}\|_\infty<\infty$ and is supported on $[-\overline\beta\pi,\overline\beta\pi]$. By \Cref{lem: d=1 constr}, for each $\gamma \in (0,2)$ the two densities satisfy
\begin{equation*}
    \|p_\epsilon-q_\epsilon\|_1\asymp \epsilon,\,\|p_\epsilon\|_{\beta,2} \asymp \|q_\epsilon\|_{\beta,2} \asymp 1,\,
    \cal{E}_\gamma(p_\epsilon, q_\epsilon) \asymp \|p_\epsilon-q_\epsilon\|_{-(1+\gamma)/2,2} \asymp \epsilon^{\frac{2\beta+\gamma+1}{2\beta}}, 
    \overline{d_H}(p_\epsilon, q_\epsilon) \asymp \epsilon^{\frac{\beta+1}{\beta}}. 
\end{equation*}
This proves both \Cref{thm:tightness} and \Cref{prop:d_H tightness} for $d=1$.
\end{proof}

\subsection{The Case $d>1$}\label{sec: lower log construction}
We move on to the case of general dimension. In \Cref{ssec:lower bd overview} we outline our approach. Then, in \Cref{ssec:lower construction} we give full details of our construction, following the argument outlined in the prior section. 
\subsubsection{Overview}\label{ssec:lower bd overview}
For the discussions below, we will assume that the ambient dimension $d\geq 2$. Our construction
here is less straightforward than for $d=1$ in \Cref{sec:1d lower construction} but shares the
same basic idea. Recall that the basic premise is that we want to saturate the H\"older's
inequality in~\cref{eq:holder}, which requires the density difference $f=\mu-\nu$ to have Fourier
transform be (almost) supported on a sphere. For $d=1$ we took $f$ to be a pure sinusoid. However, of course
such $f$ is not compactly supported and that is why we multiplied the sinusoid by a rectangle (and
then convolved many times to gain smoothness), which served as a mollifier.

For $d>1$ let us attempt to follow the same strategy and take 
$$ f_r(x) = g_r(x) h(x)\,,$$
where $r>0$ is a parameter, $h$ is some compactly supported smooth mollifier and $g_r(x)$ is defined implicitly via
\begin{equation*}
    \widehat g_r(\omega) = r^{(1-d)/2} \delta(\|\w\|-r),  
\end{equation*}
where here and below we denote, a bit
informally, by $\delta(\|\cdot\| - r)$ a distribution that  
integrates any smooth compactly supported function $\phi$ as follows:
$$ \int_{\R^d} \phi(\omega) \delta(\|\omega\| - r) \d\omega \eqdef
r^{d-1} \int_{\R^d} \phi(r\omega) \d\sigma(\omega) =
\frac{2\pi^{d/2}r^{d-1}}{\Gamma(\frac d2)} \EE[\phi(rX)]\,,$$
where $\sigma$ is the unnormalized surface measure of $\mathbb{S}^{d-1}$ and $X$ is a random
vector uniformly distributed on $\mathbb{S}^{d-1}$. Explicit computation shows
    \begin{align*}
        g_r(x) &= \cal F^{-1}[\widehat g_r](x) = \frac{\sqrt r}{(2\pi)^dr^{d/2}} \int_{\R^d} e^{i\langle\omega,x\rangle} \delta(\|\w\|-r)\D\w \\
        &= \frac{\sqrt r}{(2\pi)^{d/2}} \, \|x\|^{1-d/2} \, J_{d/2-1}(\|rx\|), 
    \end{align*}
    where $J_\nu$ denotes Bessel functions of the first kind of order $\nu$. Notice that
    $g$ is spherically symmetric and real-valued (some further properties of it are collected
    below in \Cref{lem: Bessel}).

Note that $|g_r(x)| = O(1)$ as $r\to\infty$ for any fixed $x\neq0$ (\Cref{lem: Bessel}), while at the origin we have $|g_r(0)| = \Omega(r^{(d-1)/2})$, which follows from the series expansion of the Bessel function given in for example \cite[Section 3.1-3.11]{watson1995treatise}. This causes an issue for $d>1$, as $g_r$ is too large at the origin as $r \to \infty$ compared to its tails, which makes it difficult to use it as the difference between two probability densities. Hence, we choose our mollifier $h$
to be supported on an annulus instead of on a ball. In addition, 
it will also be convenient for it to have a super-polynomially decaying Fourier transform, i.e.
$$ |\widehat{h}(w)| \leq H(\|w\|) \triangleq C \exp \left( -\frac{c \|w\|}{\log(\|w\|+2)^2}
\right) \qquad \forall w \in \R^d\,. $$
The existence of the desired function $h$ is proven \Cref{lem:final_h}.

Note that all of the Fourier energy of $g_r$ lies at frequencies $\|\omega\|=r$ by construction.
However, after multiplying by $h$ the energy spills over to adjacent frequencies as well and
we need to estimate the amount of the spill. Due to the fast decay of $\widehat h$ we will
show, roughly, the following estimates on the behavior of $\widehat f_r$:
\begin{align*}
    |\widehat f_r(\omega)| &\lesssim \tilde O(r^{(1-d)/2}) 1\{\|\omega - r\| \le \log^2(r)\} +
	r^{(d-1)/2}H(\max(\|\omega\|-r, \log^2 r)), \quad\text{and} \\
 |\widehat f_r(\omega)| &\lesssim    r^{(d-1)/2} \|\omega\|
\end{align*}
as $r\to \infty$. Note that the first bound above is super-polynomially decaying in both $\|\omega\|$
and $r$, which allows us to show that 
$$ \|f_r\|_{t,2} \le \tilde O(r^{t})\,$$for $t>-\frac{d+2}{2}$, recalling the notation $\|f\|_{t,2} = \|\|\cdot\|^t\widehat f\|_2$. 
A direct calculation will also show
$$ \|f_r\|_1 \asymp \|f_r\|_\infty \asymp \|f_r\|_2 \asymp 1\,.$$
For a desired total-variation separation $\epsilon$, we will set $\mu - \nu = \epsilon f_r$ and choose $r=\epsilon^{-1/\beta}$ to ensure that $\epsilon \|f_r\|_{\beta,2} =\tilde{O}(1)$. For the energy distance between $\mu$ and $\nu$ these choices yield
$$ \mathcal{E}_\gamma(\mu,\nu) \asymp \|\epsilon f_{\epsilon^{-1/\beta}}\|_{-{d+\gamma\over 2}, 2} = \tilde
O(\epsilon^{1 + \frac{d+\gamma}{2\beta}}) = \tilde O(\TV^{\frac{d+2\beta+\gamma}{2\beta}})\,,$$
as required.

We now proceed to rigorous details.

\subsubsection{The construction}\label{ssec:lower construction}

First, we must construct the mollifier $h$ with the properties outlined in \Cref{ssec:lower bd overview}. Recall that a function $f$ is radial (also known as spherically symmetric) if its value at $x\in\R^d$ depends only on $\|x\|$. In other words, $f(x)=f(y)$ holds for all $x,y\in\R^d$ with $\|x\|=\|y\|$.

\begin{lemma}\label{lem:final_h}
    There exists a compactly supported radial Schwartz function $h$, and a positive sequence $\{r_n\}_{n=1}^\infty$  satisfying $r_n = \Theta(n)$, such that
    \begin{align}
        \supp(h) &\subset \mbb{B}(0,1),  \label{eq:h_req_1}\\
    \supp(h) &\subset \R^d \setminus \mbb{B}(0,r_0),  \label{eq:h_req_2}\\
        |\widehat{h}(w)| &\leq C \exp \left( -\frac{c \|w\|}{\log(\|w\|+2)^2} \right) \quad \text{for all } w \in \R^d, \text{ and}\label{eq:h_req_3}\\
        \widehat{h}(r_nu)&=0 \quad \text{for all } u\in\bb S^{d-1}, \label{eq:h_req_4}
    \end{align}
    for universal constants $C, c, r_0 > 0$.
\end{lemma}

\begin{proof}First, let $h_0$ be as constructed in \Cref{lem:first_h}, which already satisfies \cref{eq:h_req_1} and \Cref{eq:h_req_3}. To address the other two requirements, we modify $h_0$ by convolving it with two additional terms:
    \[ h(x) := (A_0(\cdot) * h_0(8\cdot) * \rho_0(\cdot))(x), \]
    where $A_0$ and $\rho_0$ aim to address \Cref{eq:h_req_2} and \Cref{eq:h_req_4}, respectively, and are defined as
    \[ 
    A_0(x) = \exp\left(-\frac{1}{{1/64 - (\|x\| - 1/2)^2}}\right)\one\big\{\|x\|\in(3/8,5/8)\big\}, \quad \rho_0(x) = \one\{ \|x\| < 1/8 \}.
    \]
    Before proceeding, note that clearly $h$ is a radial Schwartz function. Let us now verify that $h$ indeed satisfies the four requirements. 
    Note that $A_0$ is an ``annulus'' supported on $\mbb{B}(0,5/8)\backslash\mbb{B}(0,3/8)$, and both $h_0(8\cdot)$ and $\rho_0$ are supported on $\mbb{B}(0,1/8)$. Therefore, $\supp(h) \subset \mbb{B}(0,7/8)\backslash\mbb{B}(0,1/8)$, which implies \Cref{eq:h_req_1,eq:h_req_2}. We now turn to the other two conditions in Fourier space. Note that 
    $$
        \widehat{h}(w) =(1/8)^d\cdot \widehat{A}_0(w) \cdot  \widehat{h}_0(w/8) \cdot \widehat{\rho}_0(w).
    $$ 
    From \Cref{item:Bessel ball} of \Cref{lem: Bessel} we know that 
    $$\mc{F} [\one\{\|\cdot\| < 1\}](w) = \left(\frac{2\pi}{\|w\|}\right)^{\frac{d}{2}} J_{\frac{d}{2}}(\|w\|).
    $$
    Hence, by \Cref{item:Bessel asymp} of \Cref{lem: Bessel}, the function $\widehat{\rho}_0(w)=(1/8)^d  \mc{F} [\one\{\|\cdot\| <
    1\}](w/8)$ has infinitely many zeros near the values of $\|w\|=8(2n\pi+\frac{(d+1)\pi}{4})$ for sufficiently large $n\in\Z^+$, which implies \Cref{eq:h_req_4}. 
    
    Finally, for \Cref{eq:h_req_3}, note that since both $A_0$ and $\rho_0$ are Schwartz functions, so are their Fourier transforms $\widehat{A}_0$ and $\widehat\rho_0$ so that
    $$
        |\widehat{h}(w)| \leq (1/8)^d \|\widehat{A}_0\|_\infty \|\widehat{\rho}_0\|_\infty |\widehat{h}_0(w/8)| \lesssim |\widehat{h}_0(w/8)|,
    $$concluding the proof.
\end{proof}

Let $h$ be as constructed in \Cref{lem:final_h}, and define 
\begin{equation}\label{eqn:f_r definition}
f_r = g_r h
\end{equation}
for $r>0$ and $\widehat g(\w) \eqdef r^{(1-d)/2} \delta(\|\w\|-r)$. Recall from the overview of our construction that we gave in \Cref{ssec:lower bd overview} that $f_r$ is our proposed density difference which we claim (approximately) saturates H\"older's inequality in \eqref{eq:holder}. The next Lemma records the properties of $f_r$ which will enable us to complete our proof. 

\begin{lemma}\label{lem:f_r norm estimate}
    Let $f_r$ be as in \eqref{eqn:f_r definition} and let $\{r_n\}_{n=1}^\infty$ be the sequence constructed in \Cref{lem:final_h}. The following hold. 
     \begin{enumerate}
        \item[(i)] For all $n\in\N$ we have
        \begin{equation*}
            \int_{\R^d} f_{r_n}(x)\d x=0 \qquad\text{and}\qquad \supp(f_{r_n})\subset \mbb{B}(0,1). 
        \end{equation*}
        \item[(ii)] We have 
        \begin{equation*}
            \|f_{r_n}\|_\infty \asymp \|f_{r_n}\|_2 \asymp \|f_{r_n}\|_1 \asymp 1, 
        \end{equation*}
        hiding constants independent of $n$.
        \item[(iii)] For any $t > -\frac{d+2}{2}$ we have 
        \begin{equation*}
            \|f_{r_n}\|_{t,2} = O(r_n^{t}\log^d(r_n))
        \end{equation*}
        as $n\to\infty$, hiding constants independent of $n$. 
        \item[(iv)] Recall the definition of $\psi_\gamma$ from \eqref{eqn:psi_gamma def}. For any $\gamma\in(0,2)$ we have
        \begin{equation*}
            \sup_{v \in \bb S^{d-1}, b\in\R} \left|\int_{\R^d} \psi_\gamma(\langle x,v\rangle-b)f_{r_n}(x)\D x\right| = O( r_n^{-(d+\gamma)/2} \log(r_n)^{d})
        \end{equation*} 
        hiding constants independent of $n$. 
    \end{enumerate} 
\end{lemma}

\begin{proof}
    Let us drop the dependence of $r_n$ to simplify notation. 

     \noindent\textbf{Showing (i).}\quad 
    Note that $\int_{\R^d} f_r(x)\d x = \widehat{f_r}(0)$. Then, $\widehat{f_r}(0)=0$ follows from the construction of $h$ and $g_r$. Indeed, $\widehat{g}_r$ is supported on $r\bb S^{d-1}$ while $\widehat h|_{r\bb S^{d-1}}\equiv0$. The fact that $\supp(f_r)\subset\mbb{B}(0,1)$ follows from $\supp(h)\subset\mbb{B}(0,1)$.
    \\
    
    \noindent\textbf{Showing (ii).}\quad
    Since $f_r$ has compact support, we immediately have $$\|f_r\|_1\lesssim \|f_r\|_2 \lesssim \|f_r\|_\infty.$$As $h$ is continuous and supported on the annulus $\{x:r_0\leq\|x\|\leq1\}$ by construction, it suffices to bound $g_r$ on said annulus. Now, for any $x$ with $r_0 \leq \|x\|\leq 1$, we have by \Cref{lem: Bessel} that
    \begin{align*}
        g_r(x)\lesssim \sqrt r\|x\|^{1-d/2}\frac{1}{\sqrt{r\|x\|}}\lesssim 1, 
    \end{align*}
    which shows that $\|f_r\|_\infty \lesssim 1$. 

    We now turn to lower bounding $\|f_r\|_1$. Recall that $h$ is uniformly continuous and nontrivial, hence $\int|h(u^*v)|\d\sigma(v)\neq 0$ for some radius $u^*$, and thus for all $u\in (u_0, u_1) \subseteq (0,1)$ for some constants $u_0, u_1$. Using that $g_r$ is spherically symmetric, we compute
    \begin{align*}
        \|f\|_1 &= \int_{\R^d} |g_r(x)| |h(x)|\d x \\
        &= \int_0^\infty u^{d-1} g_r(u,0,\dots,0) \int h(uv)\d\sigma(v)\d u\\
        &\gtrsim \sqrt r \int_{u_0}^{u_1}\left|J_{d/2-1}(ru)\right|\d u \gtrsim 1, 
    \end{align*}
    where the last line follows by \eqref{eqn:bessel at infty} once again.\\
    
    \noindent\textbf{Showing (iii).} \quad 
    Let $0 < s < r$, whose precise value will be set later. For convenience, set $B_s = \{x\in\R^d:\|x\|\leq s\}$ and $B_s^c=\R^d\setminus B_s$. Recall that by definition
    \begin{align}
        \widehat f_r(\w) &= r^{(1-d)/2} \int_{\R^d} \widehat h(\omega+x)\delta(\|x\|-r)\d x \nonumber\\
        &= \underbrace{r^{(1-d)/2} \int_{\R^d} (\widehat h\one_{B_s})(\omega+x)\delta(\|x\|-r)\d x}_{I} + \underbrace{r^{(1-d)/2} \int_{\R^d}(\widehat h\one_{B_s^c})(\omega+x)\delta(\|x\|-r)\d x}_{II}. \label{eqn:hat f_r decomp}
    \end{align}
    Let $C,c$ be as in \Cref{lem:final_h}, and $H(x) = C\exp(-c\|x\|/\log^2(\|x\|+2))$. Note that $\|\widehat h\|_\infty \leq C$. Therefore, the first term in the decomposition \eqref{eqn:hat f_r decomp} can be bounded by
    \begin{align*}
        |I| &\leq C r^{(1-d)/2} \int_{\R^d} \one\{\|\omega+x\|\leq s\}\delta(\|x\|-r)\d x\\
        &= C r^{(1-d)/2} \one\{\|\omega\|\in[r-s,r+s]\} \int_{\R^d} \one\{\|\omega+x\|\leq s\}\delta(\|x\|-r)\d x \\
        &\lesssim r^{(1-d)/2} s^{d-1} \one\{\|\omega\|\in[r-s,r+s]\}. 
    \end{align*}
    The second line uses that if $\|\omega\|\not\in[r-s,r+s]$ then the integral becomes zero. The third line uses the fact that the surface area of the intersection of $B_s$ with any sphere of any radius (and the one centered at $\omega$ with radius $r$ in particular) is at most $O(s^{d-1})$. 

    Moving on to the second term, we have
    \begin{align*}
        |II| &= r^{(d-1)/2} \int |(\widehat h \one_{B_s^c})(\omega+ru)|\d\sigma(u) \lesssim r^{(d-1)/2} H(\max\{\|\omega\|-r, s\}) 
    \end{align*}
    using that $H : [0,\infty) \to (0,C]$ is decreasing and that $|\widehat h(y)\one_{B_s^c}(y)\| \leq H(\max\{y, s\})$ for all $y\in\R^d$. Summarizing, we have the pointwise estimate
    \begin{equation}\label{eqn:f_r pointwise est}
        |\widehat f_r(\omega)| \lesssim r^{(1-d)/2}s^{d-1}\one\{\|\omega\|\in[r-s,r+s]\} + r^{(d-1)/2} H(\max\{\|\omega\|-r,s\})
    \end{equation}
    for all $\omega\in\R^d$ and $0 < s < r$. 

    We now show that $f_r$ is Lipschitz continuous. Recall from the construction of $h$ (\Cref{lem:final_h}) that $h|_{r_n\bb S^{d-1}}\equiv0$. Then, we observe that for any $\omega\in\R^d$
    \begin{align}
        |\widehat f_r(\omega)| &= r^{(d-1)/2}\left|\int \widehat h(\omega+ru)\d\sigma(u)\right|\nonumber \\
        &= r^{(d-1)/2}\left|\int \{\widehat h(\omega+ru) - \widehat h(ru)\}\d\sigma(u)\right| \nonumber\\
        &= r^{(d-1)/2} \|\widehat h\|_{\operatorname{Lip}} \frac{2\pi^{d/2}\|\omega\|}{\Gamma(\frac d2)} \nonumber \\&\lesssim r^{(d-1)/2}\|\omega\|, \label{eqn:f_r Lip}
    \end{align}
    where we use that $\widehat h$ is Schwartz by construction, and thus has finite Lipschitz constant $\|\widehat h\|_{\operatorname{Lip}}$.

    With \eqref{eqn:f_r pointwise est} and \eqref{eqn:f_r Lip} in hand we can proceed to bounding the norm of $f_r$. Let $s=D\log(r)^2$ for a large constant $D$ independent of $r$, and assume that $r$ is large enough so that $s < r/2$. Also set $\theta > 0$, whose precise value is specified later. We have 
    \begin{align*}
        \|f_r\|_{t,2}^2 &\,\,\,\,= \int_{\R^d} \|\omega\|^{2t} |\widehat f_r(\omega)|^2 \d\omega \\
        &\stackrel{\eqref{eqn:f_r Lip}}{\lesssim} r^{d-1} \int_{\|\omega\|\leq r^{-\theta}} \|\omega\|^{2t+2} \d\omega + \int_{\|\omega\| > r^{-\theta}} \|\omega\|^{2t}|\widehat f_r(\omega)|^2\d\omega \\
        &\stackrel{\eqref{eqn:f_r pointwise est}}{\lesssim} r^{d-1-\theta(2t+d)} \\
        &\qquad\qquad + r^{1-d}\log(r)^{2(d-1)} \int_{\|\omega\| > r^{-\theta}} \|\omega\|^{2t} \one\{\|\omega\|\in[r-s,r+s]\}\d\omega \\&\qquad\qquad + r^{d-1}\int_{\|\omega\| > r^{-\theta}} \|\omega\|^{2t}H^2(\max\{\|\omega\|-r,s\})\d\omega \\
        &\,\,\,\,\lesssim r^{d-1-\theta(2t+d)} + r^{2t} \log(r)^{2d-1} + r^{d-1}\int_{r^{-\theta}}^\infty u^{2t+d-1}H^2(\max\{u-r,s\})\d u.
    \end{align*}
    Note that in the derivation above we changed to polar coordinates freely, and that in the second inequality we used the assumption $t > -d/2-1$. Setting $\theta$ to any positive value greater than $(d-1-2t)/(2t+d)$ ensures that the first term in the final line is $O(r^{2t})$. As for the integral term, we can bound it by
    \begin{align*}
        &\lesssim r^{d-1} H^2(s)\int_{r^{-\theta}}^{2r} u^{2t+d-1} \d u + r^{d-1} \int_{2r}^\infty H^2(u/2)\d u \stackrel{\text{\Cref{lem: tail_int}}}{\lesssim} \operatorname{poly}(r) \times H^2(s) + r^{2t}.
    \end{align*}
    By taking $D$ large enough (independently of $r$) in the definition of $s=D\log^2(r)$ we can make also the first term $\operatorname{poly}(r)\times H^2(s)$ less than $O(r^{2t})$, which concludes the proof of $(iii)$. 
\\

    \noindent\textbf{Showing (iv).}\quad 
    The bounds that we develop below are analogous to those given in the proof of $(iii)$. Fix $b\in\R$ and $v\in\mbb{S}^{d-1}$ and define
    \begin{align*}
        \dagger:=\int_{\R^d} \psi_\gamma (\braket{v,x} -b) f_r(x)\D{x}.
    \end{align*}
    Suppose first that $\gamma \neq 1$. Then, using \Cref{lem:psi_a fourier transform stronger,lem:phi_a dct bounds}, we know by dominated convergence that 
    \begin{align*}
        \dagger
        &=\int_{\R^d} \lim_{\epsilon \to0} \int_\epsilon^{1/\epsilon} C_{\psi_\gamma} \frac{\cos(t (\braket{v,x}-b))- \one\{\gamma > 1\}}{t^{(1+\gamma)/2}} f_r(x) \D t \D{x}
        \\
        &= C_{\psi_\gamma} \lim_{\epsilon\to0} \int_\epsilon^{1/\epsilon} \Re\left\{\int_{\R^d} \frac{e^{it(\braket{v,x}-b)} f_r(x)}{t^{(1+\gamma)/2}} \D x\right\}\D t \\
        &= C_{\psi_\gamma} \lim_{\epsilon\to0} \int_\epsilon^{1/\epsilon} \frac{\cos(tb)\widehat f_r(tv)}{t^{(1+\gamma)/2}} \D t.
    \end{align*}
    Similarly, for $\gamma=1$ we can compute
    \begin{align*}
        \dagger &= C_{\psi_\gamma} \lim_{\epsilon \to 0} \int_\epsilon^{1/\epsilon} \frac{\sin(-tb)\widehat f_r(tv)}{t} \D t.
    \end{align*}
    In either case, we have $|\dagger| \lesssim \int_0^\infty |\widehat f_r(tv)|/t^{(1+\gamma)/2}\d t$. 
    
    Let $s=D\log^2(r)$ for large $D$ independent of $r$ as in the proof of $(iii)$, and let $\theta > 0$ whose precise value is specified later. Assuming that $r$ is large enough so that $s < r/2$, for any $\gamma\in(0,2)$ we have
    \begin{align}
        |\dagger| &\,\,\,\,\leq \int_0^{r^{-\theta}} \frac{|\widehat f_r(tv)|}{t^{(1+\gamma)/2}}\d t + \int_{r^{-\theta}}^\infty \frac{|\widehat f_r(tv)|}{t^{(1+\gamma)/2}}\d t \nonumber\\
        &\stackrel{\text{\eqref{eqn:f_r Lip}}}{\lesssim} r^{(d-1)/2}\int_0^{r^{-\theta}} t^{(1-\gamma)/2}\d t + \int_{r^{-\theta}}^\infty \frac{|\widehat f_r(tv)|}{t^{(1+\gamma)/2}}\d t \nonumber\\
        &\stackrel{\text{\eqref{eqn:f_r pointwise est}}}{\lesssim} r^{\frac{d-1}{2}-\theta\frac{3-\gamma}{2}} \nonumber\\&\qquad +\int_{r^{-\theta}}^\infty \frac{1}{t^{(1+\gamma)/2}}\left(r^{(1-d)/2}s^{d-1}\one\{t\in[r-s,r+s]\} + r^{(d-1)/2}H(\max\{t-r,s\})\right)\d t \nonumber\\
        &\,\,\,\,\lesssim r^{\frac{d-1}{2}-\theta\frac{3-\gamma}{2}}+r^{-(d+\gamma)/2}\log^d(r) + H(s)r^{(d-1)/2}\int_{r^{-\theta}}^{2r}\frac{\d t}{t^{(1+\gamma)/2}} + \int_{2r}^\infty \frac{H(t/2)}{t^{(1+\gamma)/2}}\d t \nonumber\\
        &\mkern-14mu\stackrel{\text{ \Cref{lem: tail_int}}}{\lesssim} r^{\frac{d-1}{2}-\theta\frac{3-\gamma}{2}}+r^{-(d+\gamma)/2}\log^d(r) + H(s)\times\operatorname{poly}(r) + r^{-100d} \label{eqn:F(v,b) bound gamma=/=1}, 
    \end{align}
    Set $\theta$ to any value greater than $(2d+\gamma-1)/(3-\gamma)$, which ensures that the first term in \eqref{eqn:F(v,b) bound gamma=/=1} is $O(r^{-(d+\gamma)/2})$. By taking $D$ large enough in the definition of $s=D\log^2(r)$, we can make $H(s)$ smaller than any polynomial in $r$, which ensures that the third term in \eqref{eqn:F(v,b) bound gamma=/=1} is also  $O(r^{-(d+\gamma)/2})$. We thus obtain the final bound $|\dagger| \lesssim r^{-(d+\gamma)/2}\log^d(r)$, concluding our proof. 
\end{proof}

\begin{proof}[Proof of Theorem \ref{thm:tightness} and Proposition \ref{prop:d_H tightness} for $d>1$] Using the functions $\{f_{r_n}\}_{n=1}^\infty$ we constructed in \Cref{lem:f_r norm estimate}, we are ready to prove \Cref{thm:tightness,prop:d_H tightness} for $d>1$. 

Let $p_0$ be a compactly supported probability density with $\inf_{\|x\|\leq1} p_0(x) > 0$. Fix the smoothness $\beta>0$. Given any desired total variation separation $\epsilon\in(0,1)$, we can find $n_0\in\N$ such that $\epsilon^{-1/\beta} \asymp r_{n_0}$, where we hide an $\epsilon$-independent multiplicative constant. Define
\begin{equation*}
    p_\epsilon = p_0 + \epsilon f_{r_{n_0}}/2 \qquad\text{and}\qquad q_\epsilon = p_0 - \epsilon f_{r_{n_0}}/2. 
\end{equation*}
Clearly $p_\epsilon$ and $q_\epsilon$ are compactly supported probability densities for all small enough $\epsilon$. Moreover, by \Cref{lem:f_r norm estimate} they satisfy
\begin{align*}
    \|p_\epsilon-q_\epsilon\|_1 \asymp \epsilon &\qquad\text{and}\qquad\|p_\epsilon\|_{\beta,2}\asymp\|q_\epsilon\|_{\beta,2}\asymp1 \qquad\text{and}\\
    \cal E_\gamma(p_\epsilon, q_\epsilon) \lesssim \epsilon^{\frac{2\beta+d+\gamma}{2\beta}} \log(1/\epsilon)^d &\qquad\text{and}\qquad \overline{d_H}(p_\epsilon, q_\epsilon)\lesssim \epsilon^{\frac{2\beta+d+1}{2\beta}}\log(1/\epsilon)^d
\end{align*}
for all fixed $\gamma\in(0,2)$. This concludes our proof. 
\end{proof}

\section{Proof of Proposition \ref{prop: tst bad}}\label{sec: proof tst bad}
\begin{proof}
    Note that $\overline{d_H}=T_{d,0}$. Let $p_\epsilon, q_\epsilon$ be the compactly supported densities constructed in the proof of \Cref{thm:tightness} in the general dimensional case. Then by construction 
    \begin{equation*}
        \eps\asymp \TV(p_\epsilon, q_\epsilon)\asymp \|p_\epsilon-q_\epsilon\|_2\quad\text{and}\quad \|p_\epsilon\|_{\beta,2} + \|q_\epsilon\|_{\beta,2}\lesssim 1\quad\text{and}\quad \overline{d_H}(p_\epsilon, q_\epsilon)\lesssim \epsilon^{\frac{2\beta+d+1}{\beta}} \log(1/\epsilon)^{d}. 
    \end{equation*}
    Write $p_{\epsilon, n}$ and $q_{\epsilon, n}$ for the empirical measures of $p_\epsilon$ and $q_\epsilon$ respectively, based on $n$ i.i.d. observations each. By the triangle inequality we have
    \begin{align*}
        \E\overline{d_H}(p_{\epsilon, n}, q_{\epsilon, n})&\leq \E\overline{d_H}(p_{\epsilon,n},p_\epsilon) + \overline{d_H}(p_\epsilon,q_\epsilon) + \E\overline{d_H}(q_\epsilon,q_{\epsilon, n}) \\ &\mkern-20mu\stackrel{\text{ \Cref{lem:VC bound}}}{\lesssim} 1/\sqrt{n}+\epsilon^{\frac{2\beta+d+1}{2\beta}}\log(1/\epsilon)^d. 
    \end{align*}
    This completes the proof. 
\end{proof}

\end{document}